\documentclass[12pt]{article}
\usepackage{anyfontsize}
\usepackage{eucal}
\usepackage{graphicx,psfrag,amssymb,amsmath,amsthm, bbm}
\usepackage[cmtip,all]{xy}
\usepackage{hyperref,comment}
\usepackage{latexsym}
\usepackage{graphicx,psfrag,amssymb,amsmath,amsthm}
 \usepackage{color}
 \usepackage{hyperref,psfrag}
 \usepackage{makeidx}
 \usepackage{eucal}\makeindex
\usepackage{pdfsync}
\usepackage{hyperref}

\usepackage{nomencl}
\usepackage{eucal}
\newtheorem{thm}{Theorem}[section]
\newtheorem{theorem}{Theorem}
\newtheorem{corollary}[thm]{Corollary}

\newtheorem{lemma}[thm]{Lemma}
\newtheorem{sublemma}[thm]{Sublemma}
\newtheorem{proposition}[thm]{Proposition}

\newtheoremstyle{mydefinition}{}{}{\normalfont}{0pt}{\scshape}{.}{.5em}{}
\theoremstyle{mydefinition}
\newtheorem{defn}{Definition}
\newtheorem{definition}[defn]{Definition}

\newtheoremstyle{myremark}{}{}{\small\normalfont}{0pt}{\small\scshape}{.}{.5em}{}
\theoremstyle{myremark}
\newtheorem{remark}{Remark}

\numberwithin{equation}{section}

\def\beq{\begin{equation}}
\def\eeq{\end{equation}}
\def\beqn{\begin{equation*}}
\def\eeqn{\end{equation*}}

\def\Lip{{\rm Lip}}

\def\bB{\mathbf{B}}

\def\bD{\mathbf{D}}
\def\bE{\mathbf{E}}
\def\bF{\mathbf{F}}
\def\bG{\mathbf{G}}

\def\bJ{\mathbf{J}}

\def\bL{\mathbf{L}}

\def\bQ{\mathbf{Q}}
\def\bR{\mathbf{R}}

\def\bc{\mathbf{c}}

\def\bh{\mathbf{h}}

\def\bk{\mathbf{k}}

\def\bn{\mathbf{n}}

\def\bp{\mathbf{p}}
\def\bq{\mathbf{q}}
\def\br{\mathbf{r}}
\def\bs{\mathbf{s}}
\def\bt{\mathbf{t}}

\def\bx{\mathbf{x}}

\def\bgamma{{\boldsymbol{\gamma}}}
\def\btheta{{\boldsymbol{\theta}}}

\def\bLambda{{\boldsymbol{\Lambda}}}

\def\btau{{\boldsymbol{\tau}}}

\def\bdelta{{\boldsymbol{\delta}}}

\def\b1{{\boldsymbol{1}}}

\def\cA{\mathcal{A}}
\def\cB{\mathcal{B}}

\def\cD{\mathcal{D}}
\def\cE{\mathcal{E}}
\def\cF{\mathcal{F}}
\def\cG{\mathcal{G}}
\def\cH{\mathcal{H}}

\def\cK{\mathcal{K}}
\def\cL{\mathcal{L}}
\def\cM{\mathcal{M}}
\def\cN{\mathcal{N}}
\def\cO{\mathcal{O}}
\def\cP{\mathcal{P}}

\def\cR{\mathcal{R}}
\def\cS{\mathcal{S}}
\def\cT{\mathcal{T}}
\def\cU{\mathcal{U}}

\def\cW{\mathcal{W}}

\def\cZ{\mathcal{Z}}

\def\IC{{\mathbb C}}

\def\IE{{\mathbb E}}

\def\IH{{\mathbb H}}

\def\IN{{\mathbb N}}

\def\IR{{\mathbb R}}

\def\IZ{{\mathbb Z}}

\def\fA{\mathfrak{A}}

\def\fF{\mathfrak{F}}

\def\hG{\tilde{G}}

\def\hW{\hat{W}}

\def\hmu{\hat{\mu}}

\def\eps{\varepsilon}
\def\hcW{\hat{\mathcal{W}}}

\def\eps{\varepsilon}

\def\ds{\displaystyle}

\def\ab{\underline{a}}
\def\cb{\underline{c}}
\def\opi{\overline{\pi}}
\def\ovarphi{\overline{\varphi}}

\def\bt{\overline{t}}
\def\tb{\underline{t}}

\def\fA{\mathfrak{A}}

\def\fF{\mathfrak{F}}

\def\hGamma{\widehat{\Gamma}}

\def\hcR{\widehat{\cR}}
\def\hF{\widehat{F}}
\def\hx{\widehat{x}}
\def\hdelta{\widehat{\delta}}
\def\hW{\widehat{W}}
\def\hcW{\widehat{\cW}}
\def\hcU{\widehat{\cU}}

\def\oT{\overline{T}}

\def\hmu{\widehat{\mu}}

\def\bLambda{{\boldsymbol{\Lambda}}}
\def\bGamma{{\boldsymbol{\Gamma}}}
\def\bOmega{{\boldsymbol{\Omega}}}
\def\btau{{\boldsymbol{\tau}}}
\def\bcK{{\boldsymbol{\cK}}}
\def\bcE{{\boldsymbol{\cE}}}
\def\bvf{{\boldsymbol{\varphi}}}
\def\beps{{\boldsymbol{\eps}}}
\def\bdelta{{\boldsymbol{\delta}}}

\begin{document}

\title{Markov partition and Thermodynamic Formalism for Hyperbolic Systems with Singularities}

\author{Jianyu Chen\thanks{Department of Mathematics and Statistics, University of Massachusetts,  Amherst MA 01003, USA. Email: jchen@math.umass.edu. }
\and Fang Wang \thanks{School of Mathematical Sciences, Capital Normal University, Beijing, 100048, China;  and Beijing Center for Mathematics and Information
Interdisciplinary Sciences (BCMIIS), Beijing 100048, China.  Email: fangwang@cnu.edu.cn.}
\and Hong-Kun Zhang\thanks{Department of Mathematics and Statistics, University of Massachusetts,  Amherst MA 01003, USA. Email: hongkun@math.umass.edu. }}

\date{\today}

\maketitle

\begin{abstract}
For 2-d hyperbolic systems with singularities, statistical properties are rather difficult to establish because of the fragmentation of the phase space by singular curves. In this paper, we  construct a Markov partition of the phase space with countable states for a general class of hyperbolic systems with singularities.
Stochastic properties with respect to the SRB measure immediately follow
from our construction of the Markov partition,
including the decay rates of correlations and the central limit theorem.
We further establish the thermodynamic formalism for the family of geometric potentials,
by using the inducing scheme of hyperbolic type.
All the results apply to Sinai dispersing billiards, and their small perturbations
due to external forces and nonelastic reflections with kicks and slips.

\vskip.5cm

\end{abstract}

\maketitle

\tableofcontents

\section{Introduction}

The study of the statistical properties of two-dimensional hyperbolic
systems with singularities is motivated mainly by mathematical
billiards with chaotic behavior, introduced by Sinai in \cite{Sin70} and
studied extensively by many mathematicians (cf.~\cite{BSC90, BSC91, C01, D01, CD, CM, Y98, Y99,
SzV04, CZ05a, BG06, C08, DSV08, CD09a, DSV09, CD09b, CZ09, CD10, BCD11, DZ11, Zh11, DZ13}).

Since the fundamental work by Hadamard and Morse, the method of  symbolic dynamics has been
widely used in the study of hyperbolic dynamical systems.
The achievements of coding dynamical systems using Markov partitions
by Sinai~\cite{Sin72}, Bowen~\cite{B75}, and Ruelle~\cite{Ru78} yield remarkably results on
the exponential decay of the correlation and other useful stochastic properties for Axiom A systems.
It is thus believed for a long time that the construction of Markov partitions is
the method to obtain effective statistical statements for more complicated hyperbolic systems.
However, Markov partitions are too delicate to construct
for the chaotic billiards and other related hyperbolic systems with singularities,
as demonstrated in \cite{BS80,BSC90}, where the Markov partitions were carefully constructed by Bunimovich, Chernov and Sinai,
because of the fragmentation of the phase space by singularities.
In particular, arbitrary short stable and unstable manifolds may be dense in the phase space.
As a consequence, even if the Markov partition exists,
it is necessarily countable with elements being products of Cantor sets. Thus advanced knowledge of the geometric structure of the elements of the Markov partition is crucial for further study of statistical properties of the dynamical system.


In this paper, for uniformly hyperbolic systems with singularities
under the standard hypotheses (\textbf{H1})-(\textbf{H5}) listed in Subsection \ref{sec: assumptions},
we are able to  construct a hyperbolic set, such that the first return to the hyperbolic set is Markov.
Moreover, the return time sequence has an exponential tail bound.
Because of the existence of singular curves,
it is a rather delicate issue to construct the special hyperbolic set,
which is necessarily a nowhere dense Cantor set.
In Theorem \ref{main}, we construct such a hyperbolic set $\cR^*$,
called a ``hyperbolic product set",
based on the concepts of the standard pairs
and standard families used by Chernov and Dolgopyat (cf.~\cite{D01, CD}),
as well as the coupling lemma proved in \cite{CM} for hyperbolic billiards
and then extended to rather general hyperbolic systems with singularities in~\cite{CZ09}.
Moreover, the Markov property of the first return to the base is due to the special property of the hyperbolic product $\cR^*$ (see Definition~\ref{def regular property}),  which is ensured by the presence of two pairs of special stable/unstable manifolds.

It has been a mathematically challenging problem to obtain the stochastic properties
for the hyperbolic systems with singularities.
The main difficulty is caused by the singularities and the resulting fragmentation of phase space,
which slows down the global expansion of unstable manifolds.
Moreover, the differential of the billiard map is unbounded and has unbounded distortion near the singularities.
In Theorem \ref{main},
we are able to construct a real Markov tower for the dynamical systems, which thus is equivalent to
a countable two-sided Markov shift from the viewpoint of symbolic dynamics.
{ A direct consequence is that the periodic orbits are dense in the basin of the SRB measure.
}Our  Markov tower  models the original uniformly hyperbolic systems with singularities, and
paves the way to understanding the ergodic and statistical properties
of stochastic processes generated by ``nice" observables.
We state the following statistical properties, which are
direct consequences of the Markov tower  structure (cf.~\cite{Y98, Y99}), as well as the coupling method (cf.~\cite{CZ09}).
\begin{itemize}
\item[(1)] Decay of Correlations:
The decay rate of correlations is closely related to the tail bound of the hitting time sequences to the tower base $\cR^*$.
That is, the exponential tail bound implies the exponential decay of correlations
for  bounded piecewise H\"older continuous functions. See Theorem \ref{thm: exp M} for precise statements.
\item[(2)] The Central Limit Theorem: Due to the exponential tail bound,
the standardization of the ergodic sum converges to the standard Gaussian distribution,
for  bounded piecewise H\"older observables (see Theorem \ref{thm: CLT M}).
\end{itemize}

The existence of Markov partition can be used to prove many other statistical properties using approaches for Gibbs-Markov maps.  Moreover, our approach can also be extended to non-uniformly hyperbolic systems with singularities,
which admit  reduced systems that satisfy Assumptions (\textbf{H1})-(\textbf{H5}), using the inducing scheme as in  \cite{CZ05a}. The research in both directions are currently under way.

Another breakthrough of this paper is that we are able to take the advantage of  the improved Markov partition , and  obtain results  on the thermodynamic formalism for hyperbolic systems with singularities. The main tools are  based on the thermodynamic formalism for the Markov shifts with countable states (cf.~Sarig \cite{Sarig99, Sarig01, Sarig03}).
The major goal in the thermodynamic formalism is to study a wide class
of natural invariant measures for a given dynamical system,
which are called the \emph{equilibrium measures} for potential functions with certain regularity,
concerning the existence, finitude and uniqueness of equilibrium measures,
as well as their ergodic and statistical properties. The equilibrium measures for the H\"older potentials had been well established for Axiom A systems
and other uniformly hyperbolic systems that admit finite Markov partitions
in the classical works of Sinai, Bowen, and Ruelle,
based on the thermodynamic formalism for the subshifts of finite type.
However, the non-compactness of the shift space becomes the main obstacle in constructing
the equilibrium measures for countable-state Markov chains.
Great efforts have been made to overcome this difficulty
(cf.~\cite{VJ62, Dob68, Gurevich69, LanRu69, Gurevich70, ADU93,
Sarig99, Sarig01, Buzzi-Sarig03, Sarig03}, etc).
Using principle results obtained by Sarig (cf.~\cite{Sarig99, Sarig01, Sarig03})
on the thermodynamic formalism for the countable Markov shifts,
Pesin, Senti and Zhang have recently developed in \cite{PS08, PSZ08, PSZ16}
a powerful framework called inducing schemes,
which are applicable to some multimodal interval maps,
Young's diffeomorphisms, the H\'enon family and the Katok map.

In this paper, using the Markov partitions obtained by the improved Markov partition , we are able to obtain the new results on thermodynamic formalisms
for uniformly hyperbolic systems with singularities with respect to the
one-parameter family of geometric potentials $\varphi_t$ given by \eqref{geo potential}.
Due to the fragmentation caused by countable singularities,  the Condition (22) used in \cite{PSZ16} fails for a large class of hyperbolic systems with singularities including most billiard systems. To overcome this difficulty, we are able to propose a new and much weaker  condition - Condition (\textbf{P*}). By taking the advantage of  the exponential tail bound of $\cR^*$, we are able
to prove Theorem~\ref{thm: equi measures}. The main results is that for a range of geometric potential, the corresponding equilibrium measure  exists and is unique. Furthermore, under
 this equilibrium measure, the system has exponential decay of correlations and satisfies the central limit theorem
for the class of locally bounded H\"older continuous observables.

 To illustrate our main theorems, we also apply the above results to Sinai dispersing billiards with finite horizon, as well as their
small perturbation due to external forces with kicks and twists,
whose ergodic and statistical properties have been widely studied in
\cite{Sin70, GO74, SC87, CH96, C01, C08, Zh11, DZ11, DZ13}.
Our main theorems provide a unified approach to study various statistical properties and the
thermodynamic formalisms for this class of billiard systems.

This paper is organized in the following way.
In Section \ref{sec: ass and results}, we list the
standard hypotheses (\textbf{H1})-(\textbf{H5}) for the hyperbolic systems with singularities,
and then state our main theorems.
In Section \ref{sec: Growth}, we introduce the concepts of standard pairs
and standard families, together with the growth lemma.
The definitions of hyperbolic product sets and magnets,
as well as the coupling lemma, are introduced in Section \ref{sec: magnet}.
We prove our first main theorem - Theorem \ref{main} in Section \ref{sec: main proof},
where the coupling lemma plays an important role.
In this section, we also detect periodic points via the symbolic coding,
which helps us to obtain a suitable hyperbolic product set as the base of Markov partition.
Our first main theorem implies several advanced statistical properties,
which we shall provide in Section~\ref{sec:stat}.
In Section \ref{sec TF IS}, we recall necessary preliminaries on the thermodynamic formalisms
for the countable Markov shifts, and prove our second main theorem - Theorem~\ref{thm: equi measures}, by
applying the inducing scheme on the tower base.
In Section \ref{sec: app}, Theorem \ref{thm: C1} - the application of our main results to the Sinai billiards and their small perturbations - is proven by verifying
Assumptions (\textbf{H1})-(\textbf{H5}) and Condition (\textbf{P*}).

\section{Assumptions and Main Results}\label{sec: ass and results}

In this paper we study hyperbolic systems with singularities.
Let $M$ be a smooth Riemannian surface (possibly with boundary) and $T: M\to M$ be a time-reversible map with singularities, see Assumption (\textbf{H2}) below for details.

We denote by $d(\cdot,\cdot)$ the geodesic distance between two points on $M$.
Given a smooth curve $W\subset M$, we denote by $d_W(\cdot, \cdot)$ the distance on $W$ under the induced Riemannian metric on $W$. Furthermore, let $m_W$ be the Lebesgue measure on $W$ under $d_W(\cdot, \cdot)$,
and $|B|:=m_W(B)$ be the corresponding ``length" for any measurable subset $B\subset W$.

\subsection{Assumptions.}\label{sec: assumptions}

In the following we list and briefly explain the assumptions for our main theorems.

\bigskip

\noindent \textbf{(H1)  Uniform hyperbolicity.}
There exist two families of continuous cones $C^u_x$ (unstable) and $C^s_x$
(stable)  in the tangent spaces $\cT_x M$, for all $x\in M$,
with the following
properties:
\begin{itemize}
\item[(1)] $D_x T (C^u_x)\subset C^u_{ T x}$ and $D_x T (C^s_x)\supset C^s_{ T x}$  whenever $DT$ exists.
\item[(2)] There exists a constant $\Lambda>1$ such that
$$
\|D_x T(v)\|\geq \Lambda \|v\|,\ \forall  v\in C_x^u, \ \ \text{and}\ \
\|D_xT^{-1}(v)\|\geq \Lambda \|v\|,\  \forall v\in C_x^s.
$$
\item[(3)] The angle between $C^u_x$ and $C^s_x$ is uniformly bounded away
from zero.
\end{itemize}

\begin{definition}
A smooth curve $W\subset M$ is said to be an unstable (resp. stable) curve,
if the tangent line $\cT_x W$ belongs to the unstable (resp. stable) cone $C^u_x$ (resp. $C^s_x$) at every point $x \in W$.
As usual,
a curve $W\subset M$ is called an unstable (resp. stable) manifold if $T^{-n}W$ is an unstable (resp. stable) curve for all $n\ge 0$ (resp. $n\le 0$).
\end{definition}

\bigskip

\noindent {\textbf{(H2) Singularities}.}
Let $\cS_0\subset M$ be a closed subset containing $\partial M$.
We set $\cS_{\pm 1}=\cS_0\cup T^{\mp 1}\cS_0$, and call them the singularity set of $T^{\pm 1}$ respectively.

\begin{enumerate}
\item[(1)]
The map $T: M\backslash \cS_1 \to M\backslash \cS_{-1}$ is a $C^{1+\bgamma_0}$ diffeomorphism
for some $\bgamma_0\in (0, 1]$.
Moreover, $T$ can be extended  by continuity on the closure of each component of $M\backslash \cS_1$.\footnote{The extensions on different components may be different.}

\item[(2)]
Both $\cS_1$ and $\cS_{-1}$ consist of finite or countably many compact smooth curves.
Moreover, curves in $\cS_0$ are uniformly transverse to both stable and unstable cones,
and curves in $\cS_1\backslash \cS_0$ (resp. $\cS_{-1}\backslash \cS_0$) are stable (resp. unstable) curves.
Each curve in $\cS_{1}$ (resp. $\cS_{-1}$) terminates either inside another curve of
{ $\cS_{1}\backslash \cS_0$ (resp. $\cS_{-1}\backslash \cS_0$) or on $\cS_0$}.

\item[(3)] There exist $\bq_0\in (0, 1)$, $\bs_0\in (0,1]$ and $C>0$, such that for any $x\in M\setminus \cS_1$,
     \beq\label{def q0}
     \|D_x T \|\leq C  d(x, \cS_1)^{-\bq_0}.
     \eeq
     Moreover,  the $\eps$-neighborhood of the singularity set $\cS_1$ is not too  heavy under $m$:
  \beq\label{def s0}
  m(\{x\in M\,:\, d(x, \cS_{\pm 1})\leq
    \eps\})\leq C\cdot \eps^{\bs_0},\ \ \forall\ \eps>0 .
  \eeq
\end{enumerate}

\bigskip

It follows from \cite{KS86} that condition (\ref{def q0}) implies that there exist stable manifold $W^s(x)$ and unstable manifold $W^u(x)$
for $m$-a.e. $x\in M$.
For any $n\geq 1$, let $\ds\cS_{\pm n}=\cup_{m=0}^{n-1} T^{\mp m}\cS_{\pm 1}$
be the singularity set of $T^{\pm n}$,
and $\ds\cS_{\pm \infty}=\cup_{m\geq 0} \cS_{\pm m}$.
Note that any maximal stable/unstable manifold $W^{s/u}$ is an open connected curve in $M\backslash \cS_{\pm \infty}$,
usually with two endpoints in $\cS_{\pm \infty}$.\footnote{
In the paper, $W^{s/u}$ always stand for the maximal stable/unstable manifolds.
}
Assumption \textbf{(H2)}(2) implies that the angles between both stable and
unstable manifolds with the singular curves in
$\cS_{\pm 1}$ is uniformly bounded away from zero at their intersection points.\\


A $T$-invariant Borel probability measure $\mu$ on $M$ is
said to be an {\it SRB measure}  (short for Sinai-Ruelle-Bowen),
if the conditional measure of $\mu$ on each unstable manifold $W^u$ is
absolutely continuous with respect to the Lebesgue measure $m_{W^u}$.
Under our  assumptions (\textbf{H1})-(\textbf{H2}) and (\textbf{H4}) (stated below),
the {\it existence} and \emph{finitude} of SRB measures can be derived by
standard arguments, see~\cite{KS86, P92, Sataev92, Y98, ABV, CD, DZ11, DZ13} for references.
Here the finitude means that there are finitely many ergodic SRB measures, and each of them
is mixing up to a cyclic permutations.
\\

\noindent {\textbf{(H3) SRB measure}.} There is a mixing SRB measure $\mu$ for the system $(M, T)$.
\\

As $\mu$ might not be the unique SRB measure, from now on,
whenever we pick initial points or stable/unstable curves,
we automatically take them from the basin of $\mu$ without emphasizing.

\bigskip

\begin{definition}
Given two points $x, y \in M$. Let $\bs_+(x,y)$ be the smallest integer
$n\geq 0$ such that $x$ and $y$ belong to distinct components of
$M\setminus\cS_n$, which is called the \emph{forward separation time} of $x, y$.
The \emph {backward separation time $\bs_-(x,y)$} is defined in a similar fashion.
\end{definition}

\noindent {\textbf{(H4) Regularity of stable/unstable manifolds}.}
We assume there exist a family of stable manifolds $\cW^s$ and a family of unstable manifold  $\cW^{u}$,
which are invariant under $T^{\pm}$.
The families $\cW^{s/u}$ are regular in the following sense:
\begin{itemize}
\item[(1)] \textit{Bounded curvature and length.}
There exist $B>0$ and $\bc_M>0$ such that for any $W\in \cW^{u/s}$, the curvature of $W$ is bounded
from above by $B$, and the length $|W|\le \bc_M$.\footnote{
Bounded length can be obtained by adding auxiliary singular curves, if necessary.
}

\item[(2)] \textit{Bounded distortion of $T$.} There exist
     $C_{\bJ}>1$
     and $\bgamma_0\in(0,1]$ such that for each unstable
      manifold $W\in \cW^u$ and each pair of points $x,\ y\in W$,
 \beq
	  \left|\,\log J_W T(x)-\log J_W T(y)\right|
		\leq	C_{\bJ}\, d(x,y)^{\bgamma_0}\label{distor0},
 \eeq
 where $J_W T(x)$ is  the Jacobian of $T$ at $x\in W$ with respect to the Lebesgue measure on $W$.\footnote{
     The subscript $\bJ$ in the constant $C_{\bJ}$ stands for Jacobian.}

\item[(3)]\textit{Absolute continuity.} For each pair of regular
unstable manifolds $W^1$ and $ W^2$, which are close to each other, we define
\beq\label{abs cts domain}
   W^i_{\ast} := \{ x\in W^i \colon
   W^s( x) \cap W^{3-i} \neq \emptyset\}, \ \ \text{for} \ i=1, 2.
\eeq
The holonomy map $\bh^s:W^1_{\ast}\to W^2_{\ast}$
along stable manifolds is absolutely continuous, with its
Jacobian $J_{W^1_*}\bh^s$ uniformly bounded from above and away from zero.
Furthermore, there exists $ \btheta\in(0,1)$ such that for any $x,\ y\in W^1_*$,
 \beq\label{cjchb}
 |\log J_{W^1_*}\bh^s(x)-\log J_{W^1_*}\bh^s(y)|\leq
 C_{\bJ}\cdot\btheta^{\bs_+(x,y)}.
 \eeq
\end{itemize}



\bigskip

Because of the unbounded derivative of the map (\ref{def q0}), we may need to add extra singular curves to guarantee the distortion bound. A useful strategy is to add extra homogeneity strips,  see for example (\ref{homogeneity}); which was first used by Sinai \cite{Sin70}.
Since the singularity sets $\cS_{\pm 1}$ may contain countably infinitely many curves,
some of the stable/unstable manifolds may be relatively short.
The following condition ensures that
a large portion of these invariant manifolds are not that short after sufficiently many iterates.
See Lemma \ref{lem: global} for a quantitative estimate.\\

\noindent {\textbf{(H5) { One-step expansion.}} Given an unstable manifold $W\subset M$, we let
$\{V_\alpha\}$ be the collection of connected components of $TW\backslash \cS_{-1}$
and set $W_\alpha=T^{-1}V_\alpha$. We have the following assumption:
\beq
\liminf_{\delta_0\to 0}\ \sup_{W\colon
  |W|<\delta_0} \ \sum_{V_{\alpha}}
 \left(\frac{|W|}{|V_{\alpha}|}\right)^{{\bs_0}}
 \frac{|W_{\alpha}|}{|W|}<1,
\eeq
where the supremum is taken over all unstable manifolds $W\subset M$, and $\bs_0\in (0,1]$ is given in (\ref{def s0}).
\\

Next we specify the class of observables that we will use to investigate the statistical properties, especially the decay rates of correlations.
For any $\gamma\in (0,1)$, we consider those bounded, real-valued  functions $f\in L^{\infty}(M,\mu)$ such that, there exists a measurable foliation $\cW_f^s$ of $M$ into stable curves, with the property that  for any $x$ and $y$ lying on one stable curve $W\in \cW_f^s$,
\beq \label{sDHC-} 	|f(x) - f(y)| \leq \|f\|^-_{\gamma,\cW^s_f} d (x,y)^{\gamma},\eeq
with
$$\|f\|^-_{\gamma,\cW^s_f}\colon = \sup_{W\in \cW^s_f,}\sup_{ x, y\in W}\frac{|f(x)-f(y)|}{ d (x,y)^{\gamma}}<\infty.$$
Note that there may exist several measurable partitions $\cW^s_{f,\alpha}$ of $M$ into stable curves, such that (\ref{sDHC-}) holds, for $\alpha$ belonging to an index set $\cA_f$. We denote these partitions as $$\cW^-_f:=\{\cW^s_{f,\alpha},\alpha\in \cA_f\}.$$
We also require that the stable foliation $\cW^s$ of $M$ belongs $\cW^-_f$.
Now we define
$$\|f\|^-_{\gamma}\colon =\sup_{\alpha\in \cA_f} \|f\|^-_{\gamma,\cW^s_{f,\alpha}}.$$

Let $\cH^-(\gamma)$ be the collection of all such observables $f$, such that $\|f\|^-_{\gamma}<\infty$. Then for any $f\in \cH^-(\gamma)$, any stable curve $W\in \cW^-_f$, we have
\beq \label{DHC-} 	|f(x) - f(y)| \leq \|f\|^-_{\gamma}  d (x,y)^{\gamma}.\eeq

Similarly, we define $\cH^{+}(\gamma)$ as the set of all bounded, real-valued  functions $g\in L^{\infty}(M,\mu)$  such that,  for any $g\in \cH^+(\gamma)$, any unstable curve $W\in \cW^+_g$, we have
\beq \label{DHC+} 	|g(x) - g(y)| \leq \|g\|^+_{\gamma}  d (x,y)^{\gamma}.\eeq
Here $\cW^+_g=\{\cW^u_{g,\alpha},\, \alpha\in \cA_g\}$ is the collection of measurable partitions $\cW^u_{g,\alpha}$ of $\cM$ into unstable curves $\cW^u_g$, such that the unstable foliation $\cW^u$ of $M$ belongs $\cW^+_g$. Moreover,
for any $x$ and $y$ lying on one unstable curve $W \in \cW^{u}_{g,\alpha}$,
\beq \label{DHC+1} 	
   |g(x) - g(y)| \leq \|g\|^+_{\gamma,\cW^u_{g,\alpha}} d (x,y)^{\gamma},
\eeq
where
$$
   \|g\|^+_{\gamma,\cW^u_{g,\alpha}}\colon = \sup_{W\in \cW^u_{g,\alpha}}\sup_{ x, y\in W}\frac{|g(x)-g(y)|}{ d (x,y)^{\gamma}}<\infty.
$$
 Thus we have that
$$\|g\|^+_{\gamma}\colon =\sup_{\alpha\in \cA_g} \|g\|^+_{\gamma,\cW^u_{g,\alpha}}<\infty.$$

For every $f\in \cH^{\pm}(\gamma)$ we define
\beq \label{defCgamma}
   \|f\|^{\pm}_{C^{\gamma}}\colon=\|f\|_{\infty}+\|f\|^{\pm}_{\gamma}.
\eeq
In particular, if $f$ is H\"{o}lder continuous on every component of $M\setminus \cS_k$, for some integer $k\in \mathbb{Z}$, with H\"{o}lder exponent $\gamma$, then one can check that $f\in \cH^{\pm}(\gamma)$. To study the statistical properties of $f\circ T^n$, we usually require $f\in H(\gamma):=H^+(\gamma)\cap H^-(\gamma)$.

\subsection{Statement of Results}\label{Sec: Generalized Markov partition }

In this paper, we always assume that the map $T: M\to M$ satisfies
(\textbf{H1})-(\textbf{H5}) given in Subsection~\ref{sec: assumptions}.

The first main result in this paper is the existence of a countable Markov partition for the system $(M, T)$.
Moreover, the measure of elements of the partition has exponential tail bounds.
Our first main theorem is stated as follows.

\begin{theorem}\label{main}
There is a hyperbolic product set $\cR^*\subset M$ of positive $\mu$-measure
such that the following properties hold:
\begin{enumerate}
\item[(1)] \textbf{Decomposition of $\cR^*$:}
there is a countable family of mutually disjoint (mod $\mu$) closed $s$-subsets\footnote{See Definition \ref{defnsset} for s/u-subset.} $\cR_n^*\subset \cR^*$ such that
$\cR^*=\cup_{n\geq 1} \cR_n^* \pmod \mu$.\footnote{
Some $\cR_n^*$ can be null, i.e., $\mu(\cR_n^*)=0$.
}

\item[(2)] \textbf{First return is proper:} for any $\cR_n^*$ with $\mu(\cR_n^*)>0$, $T^n\cR_n^*$ returns to $\cR^*$ for the first time with the following properties:
  \begin{itemize}
  \item[(i)]  $T^{n} \cR_n^*$ is a $u$-subset of $\cR^*$;
  \item[(ii)] $T^i\cR_n^*\cap\cR^*=\emptyset \pmod \mu$  for any $i=1, 2, \cdots, n-1$.
  \end{itemize}

\item[(3)] \textbf{Exponential tail bound:}
there are $C_*>0$ and $\theta_*\in (0, 1)$ such that
    \beq\label{cRmbd}
    \sum_{k\geq n} \mu(\cR_k^*)\leq C_* \theta_*^{n}, \ \ \ \  \forall n\ge 1.
    \eeq
\end{enumerate}
As a result, the following collection of disjoint measurable sets
\beq\label{def Markov Partition}
\cP:=\{T^k \cR_n^*: \  n\in \IN, \ k=0, \dots, n-1 \}
\eeq
forms a Markov partition with respect to the SRB measure $\mu$. 
\end{theorem}

The precise definitions of the hyperbolic product set and the $s/u$-subsets will be given in Section \ref{sec: magnet}.

According to the above theorem,  the original system $T: M\to M$
with respect to the SRB measure $\mu$ is conjugate to a Markov shift  with countable states.
{ An immediate consequence is the density of periodic orbits. More precisely, we have

\begin{theorem} The set of periodic orbits for the system $T: M\to M$
is dense in the basin of the SRB measure $\mu$.
\end{theorem}
On the other hand,}
many existing results on statistical properties for  hyperbolic systems with Markov partition
and Gibbs-Markov maps would hold for our system $(M, T, \mu)$,
{ due to the exponential tail bound property in \eqref{cRmbd}.
}Indeed by using the coupling Lemma \cite{CZ09},
we also obtain the following upper bounds for the rate of decay of correlations.

\begin{theorem}
\label{thm: exp M}
The system $(M, T,\mu)$ enjoys exponential decay of correlations, that is, for $\gamma\in (0,1)$,
for any $f\in\cH^-(\gamma)$ and $g\in\cH^+(\gamma)$ on $M$, we have
\beq\label{Cn}
\left|\int_{M} f\cdot\, g\circ T^n\, d\mu - \int_{M} f\, d\mu  \int_{M} g\, d\mu \right| \le C_{f, g}\ \theta^n,
\eeq
where $C_{f, g}$ is a constant that depends on observables $f$ and $g$, and $\theta=\theta(\gamma)\in(0,1)$.
\end{theorem}
Moreover, it follows from \cite{Y98}, we also have the following Central Limit Theorem.
\begin{theorem}
\label{thm: CLT M}
Let $f$ be bounded, and  H\"{o}lder continuous on every component of $M\setminus \cS_k$, for some integer $k\in \mathbb{Z}$, with H\"{o}lder exponent $\gamma$, and $\mu(f)=0$. Then $\{f\circ T^n, n\geq 0\}$ satisfies the Central Limit Theorem, that is,
$$
\frac{1}{\sqrt{n}} \sum_{k=0}^{n-1} f\circ T^k\overset{dist}{\longrightarrow} N(0, \sigma^2),
\ \ \text{as}\ n\to \infty,
$$
where $\sigma^2=\sum\limits_{k=-\infty}^\infty \int f\cdot\, f\circ T^k\, d\mu$.
The degenerate case $\sigma=0$ occurs if and only if $f$ is a coboundary, i.e. $f=g-g\circ T$ for some $g\in L^2_\mu(M)$.
\end{theorem}

{
Besides the exponential decay of correlations and the central limit theorem,
we shall provide more advanced statistical properties in Section~\ref{sec:stat},
including the large/moderate deviation principles, the functional/local central limit theorem,
the concentration inequalities,  etc.
}

\vskip.2in

Another major goal of this paper is to study the thermodynamic formalisms
for the uniform hyperbolic system $(M, T)$ with singularities,
with respect to the one-parameter family of geometric potentials defined by
\begin{equation}\label{geo potential}
\varphi_t(x)=-t\log |J^u T(x)|,  \ -\infty < t< \infty,
\end{equation}
where $J^u T$ is the Jacobian of $T$ along the unstable manifolds with respect to Lebesgue measure.
We focus on the existence and uniqueness of the corresponding equilibrium measures,
as well as their statistical properties.

A powerful framework called ``inducing schemes" has
recently been developed by Pesin, Senti and Zhang in \cite{PS08, PSZ08, PSZ16}.
In particular, this method applies to the Markov partition s of hyperbolic type,
under an additional condition - Condition (22) in \cite{PSZ16}.
However,  for hyperbolic systems with singularities, Condition (22) in \cite{PSZ16} usually fails. Mainly because  the differential of the system $DT$ can be unbounded, and even worse,  the singular set $\cS_1$ can have infinitely many singular curves.
To overcome these difficulties, we propose a much weaker condition - Condition (\textbf{P*}) given by equation \eqref{cond left}, which is suitable
for a large class of billiard systems and their small perturbations.
Under Condition (\textbf{P*}) ,
we adapt the inducing scheme framework and
perform a detailed analysis on the hyperbolic product set $\cR^*$ constructed in Theorem \ref{main}.

\begin{theorem}\label{thm: equi measures}
Assume that Condition (\textbf{P*}) given by \eqref{cond left} holds, then
\begin{enumerate}
\item[(1)]  There are $\tb<1<\bt$ such that for any $t\in (\tb, \bt)$,
there is a unique equilibrium measure $\mu_t\in \cM_L(T, M)$ for the potential $\varphi_t$.
\item[(2)] $\mu_t$ has exponential decay of correlations,
that is,
for any $f\in\cH_\pi(r)$ and $g\in\cH_\pi(r)$, where the class $\cH_\pi(r)$ is given by \eqref{def cH pi},
we have
\beq\label{Cn1}
\left|\int_{M} f\cdot\, g\circ T^n\, d\mu_t - \int_{M} f\, d\mu_t  \int_{M} g\, d\mu_t \right| \le C_{f, g}\ \theta^n,
\eeq
where $C_{f, g}$ is a constant that depends on observables $f$ and $g$, and $\theta=\theta(r)\in(0,1)$.
\item[(3)]
Let $f\in \cH_\pi(r)$ such that $\mu_t(f)=0$. Then $\{f\circ T^n, n\geq 0\}$ satisfies the Central Limit Theorem, that is,
$$
\frac{1}{\sqrt{n}} \sum_{k=0}^{n-1} f\circ T^k\overset{dist}{\longrightarrow} N(0, \sigma_t^2),
\ \ \text{as}\ n\to \infty,
$$
where $\sigma_t^2=\sum\limits_{k=-\infty}^\infty \int f\cdot\, f\circ T^k\, d\mu_t$.
The degenerate case $\sigma_t=0$ occurs if and only if $f$ is a coboundary, i.e. $f=g-g\circ T$ for some $g\in L^2_{\mu_t}(M)$.
\end{enumerate}
\end{theorem}

The definitions of equilibrium measures and the class $\cM_L(T, M)$ of liftable measures
will be given in Subsection~\ref{sec thermo IS}, and
the technical condition (\textbf{P*}) will be given in Subsection~\ref{sec TF M}.

We point out that the equilibrium measure $\mu_1$ that we obtained in Theorem \ref{thm: equi measures}
is exactly the SRB measure $\mu$ given in Assumption (\textbf{H3}).
Also, the uniqueness of equilibrium measures is only among the class $\cM_L(T, M)$,
which may not contain all $T$-invariant probability measures on $M$.
\\

Our main results provide a unified approach
to study various statistical properties and thermodynamic formalisms
for a large class of hyperbolic systems with singularities, which including the dispersing billiards and its small perturbations. The ergodic and statistical properties of
such billiard maps and its perturbations have been widely studied
in \cite{Sin70, GO74, SC87, CH96, C01, C08, Zh11, DZ11, DZ13}. Here, to make the demonstration  of our results more transparent, we in particularly apply our results to dispersing billiards with finite horizon, as well as their
small perturbation due to external forces with kicks and twists, which is denoted as the class $\cF(\bQ, \btau_*, \beps_*)$ of billiard maps
on a fixed table $\bQ$. For the precise definition of this class of maps, see details in Subsection \ref{sec: Lorentz setup}.

\begin{theorem}\label{thm: C1}
For any map $T\in \cF(\bQ, \btau_*, \beps_*)$, results of Theorem~\ref{thm: equi measures}
hold for the geometric potential $\varphi_t$ with $t$ inside a neighborhood of $1$.
\end{theorem}

 Indeed it was verified in \cite{DZ11, DZ13} that any  map $T\in \cF(\bQ, \btau_*, \beps_*)$ satisfies Assumptions (\textbf {H1})-(\textbf{H5}).
Moreover, if $\beps_*$ is sufficiently small, then later in Lemma 7.1  we check that
Condition \eqref{cond left} holds for any $t_0\in (1/2, 1)$.
Consequently, the results in Theorem~\ref{thm: equi measures}
hold for the geometric potential $\varphi_t$ with $t$ inside a neighborhood of $1$.

\section{Standard Families and Growth Lemmas}\label{sec: Growth}

\subsection{Standard Pairs and Standard Families}

In this section we first introduce the concepts of standard pairs and standard families. The idea of standard pairs was first brought up by Dolgopyat in \cite{D01}, and then extended to more general systems by Chernov and Dolgopyat in \cite{C06,CD}.

For our purpose, the underlying curve for a standard pair is taken to be an unstable manifold in $\cW^u$,
the $T$-invariant family in Assumption (\textbf{H4}), and the corresponding reference measure is
the conditional probability measure $\mu_W$ of the SRB measure $\mu$ in Assumption (\textbf{H3}).
It is well known that $\mu_W$ is absolutely continuous with respect to the Lebesgue measure $m_W$, and
its density function $\rho_{W}$ is uniquely determined by
 \beq\label{dens}
 \frac{\rho_{W}(y)}{ \rho_{W}(x)} =
\lim_{n\rightarrow \infty}\frac{J_{W}T^{-n}(y)} {J_{W}T^{-n}(x)}.
\eeq
It follows from the distortion bound (\ref{distor0}) that $\rho_W\sim |W|^{-1}$ on $W$.

\begin{definition}[Standard Pair]
Let $C_{\bJ}>0$ and $\bgamma_0\in (0,1)$ be the constants given in (\ref{distor0}).
Fix a large constant $C_{\br}>C_{\bJ}/(1-1/\Lambda)$.

Given an unstable manifold $W\in \cW^u$ and a probability measure $\nu$ on $W$, the pair
$(W,\nu)$ is said to be a \textit{standard pair} if $\nu$ is absolutely
continuous with respect to $\mu_W$,
such that the density function
$g_W(x):=d\nu/d\mu_W (x)$ is regular in the sense that
 \beqn\label{lnholder0}
 | \log g_W(x)- \log g_W(y)|\leq C_{\br}\cdot d_W(x,y)^{\bgamma_0}.
 \eeqn
\end{definition}

Iterates of standard pairs require the following extension of standard pairs.

\begin{definition}[Standard family]
Let $\cG=\{(W_{\alpha}, \nu_\alpha),\ \alpha\in \cA, \ \lambda\}$ be a family of standard pairs
equipped with a factor measure $\lambda$ on the index set $\cA$.
We call $\cG$ a {\it standard family} if the following conditions hold:
\begin{enumerate}
\item[(i)] $\cW:=\{W_{\alpha}: \alpha\in \cA\}$ is a measurable partition into unstable manifolds in $\cW^u$;
\item [(ii)] There is a finite Borel measure $\nu$ supported on $\cW$ such that for any measurable set $B\subset M$,
$$
\nu(B)=\int_{\alpha\in\cA} \nu_{\alpha}(B\cap W_{\alpha})\, d\lambda(\alpha).
$$
\end{enumerate}
For simplicity, we denote such a family by $\cG=(\cW, \nu)$.
\end{definition}

Given a standard family $\cG=(\cW, \nu)=\{(W_{\alpha}, \nu_\alpha),\ \alpha\in \cA, \ \lambda\}$ and $n\ge 1$,
the forward image family $T^n\cG=(\cW_n, \nu_n)=\{(V_\beta, \nu^n_\beta), \ \beta\in \cB, \ \lambda^n\}$
is such that each $V_\beta$ is a connected component of $T^nW_\alpha\backslash \cS_{-n}$
for some $\alpha\in \cA$,
associated with
$\nu^n_\beta(\cdot)=T^n_*\nu_\alpha(\cdot \ |\ V_\beta)$ and
$d\lambda^n(\beta)=\nu_\alpha(T^{-n}V_\beta) d\lambda(\alpha)$.
It is not hard to see that $T^n\cG$ is a standard family as well (cf. \cite{CZ09}).

\subsection{Growth Lemmas}

Let $\fF$ be the collection of all standard families.
We introduce a characteristic function on $\fF$
to measure the average length of unstable manifolds in a standard family.
More precisely,
we define a function $\cZ: \fF\to [0, \infty]$ by
\beq\label{cZ}
\cZ(\cG)=\dfrac{1}{\nu(M)} \int_{\cA} |W_{\alpha}|^{-\bs_0} \, d\lambda(\alpha), \ \  \text{for any}\  \cG\in \fF,
\eeq
where $\bs_0$ is given by \eqref{def s0}.
Let $\fF_0$ denote the class of $\cG\in \fF$ such that $\cZ(\cG)<\infty$.

It turns out that the value of $\cZ(T^n\cG)$ decreases exponentially in $n$ until it becomes small enough.
This fundamental fact, called the growth lemma, was first proved by Chernov for dispersing billiards in \cite{C99},
and later proved in \cite{CZ09} under Assumptions (\textbf{H1})-(\textbf{H5}).

\begin{lemma}\label{lem: growth ch}
There exist constants $c>0$, $C_z>0$, and $\vartheta\in (0,1)$ such that for any $\cG\in \fF_0$,
$$
\cZ(T^n \cG)\leq c \vartheta^{n} \cZ(\cG)+C_z, \ \ \ \text{for any} \ n\ge 0.
$$
\end{lemma}

For any standard family $\cG=(\cW, \nu)$,
any $x\in W\in \cW$ and $n\ge 0$,
we denote by $r_{\cG, n}(x)$ the distance from $T^n x$ to $\partial W_n$ measured along $W_n$,
where $W_n$ is the open connected component of $T^n W$ whose closure contains $T^n x$.\footnote{
If $T^n x\in \partial W_{n}$, we set $r_{\cG, n}(x)=0$.
}
As a result of Lemma \ref{lem: growth ch},
we have

\begin{lemma}[cf. \cite{CZ09}]\label{lem: growth}
There exists $c_1>0$ such that for any standard family $\cG=(\cW, \nu)\in \fF_0$, any $\eps>0$ and $n\ge 0$,
we have
\beq
\nu(r_{\cG, n}(x)<\eps)\le (c_1 \cZ(\cG) + C_z) \eps^{\bs_0}.
\eeq
\end{lemma}

Next we measure the size of maximal stable and unstable manifolds.
Given $x\in M$, we denote by $r^{s/u}(x)=d^{s/u}(x, \cS_{\pm \infty})$ the distance from $x$ to the nearest endpoint of the maximal manifolds
$W^{s/u}(x)$.\footnote{
If $W^{s/u}(x)$ does not exist, we set $r^{s/u}(x)=0$. }
We set $r_n(x)=r^u(T^n x)$ for any $n\ge 0$.
It is clear that $r_{\cG, n}(x)\le r_n(x)$ for any standard family $\cG=(\cW, \nu)$,
where the equality holds if the connected component in $\cW$ that contains $x$ is a maximal unstable manifold.
For any small $\delta>0$, we denote
\beq
N_{\delta}:=\bigcap_{n=0}^\infty \left\{x\in M: \ r_{n}(x)\ge \delta \Lambda^{-n(1-\bq_0)} \right\}.
\eeq

\begin{lemma}[cf. \cite{CZ09}]\label{lem: global} The following statements hold:
\begin{itemize}
\item[(1)] There exist $c_2>0$ and $c_3>0$ such that
for any standard family $\cG=(\cW, \nu)\in \fF_0$, we have
$$
\dfrac{\nu(N_{\delta})}{\nu(M)}\ge 1-  \delta^{\bs_0} (c_2 \cZ(\cG) +c_3).
$$
\item[(2)] In addition,
there exists $c_4>0$ such that for any $x\in N_{\delta}$, we have
$$
r^u(x)= r_{0}(x)\ge \delta, \ \ \text{and}\ \
r^{s}(x)\ge \inf_{n\ge 0}\ c_4\Lambda^n \left(r_{n}(x)\right)^{\frac{1}{1-\bq_0}} \ge c_4\delta^{\frac{1}{1-\bq_0}}.
$$
\end{itemize}
\end{lemma}

 Item (1) of this lemma was shown in \cite{CZ09} for all proper standard families (which we shall introduce at the end of this subsection), and can be easily modified for all standard families in $\fF_0$. Item (2) follows from time reversibility of the map.
Lemma \ref{lem: global} will be a key tool to construct a magnet in Section \ref{sec: R star}.

It was shown in \cite[Lemma 12]{CZ09} that the $T$-invariant standard family
$\cE=(\cW^u, \mu)$ belongs to $\fF_0$,
where $\cW^u$ is the collection of all maximal unstable manifolds
and $\mu$ is the SRB measure.

Fix a large constant $C_\bp>0$ such that
\beq\label{proper constant 1}
C_\bp>\max\{1+C_z, \ \cZ(\cE), \ (1/\hdelta)^{\bs_0}\},
\eeq
where $\hdelta$ is given by \eqref{proper constant 2}.

\begin{definition}
A standard family $\cG\in \fF_0$ is said to be \emph{proper} if $\cZ(\cG)<C_\bp$.
\end{definition}

By Lemma \ref{lem: growth ch}, any standard family $\cG\in \fF_0$ will eventually
become proper after sufficiently many iterates.
In particular, $\cE=(\cW^u, \mu)$ is a proper standard family.

\section{Hyperbolic Product sets, Rectangles and the Coupling Lemma}\label{sec: magnet}

\subsection{Hyperbolic Product Sets and Rectangles}

In this subsection, we introduce the basic blocks, the so-called hyperbolic product sets, for our construction (see \cite{CZ09} and \cite[\S ~7.11-7.12]{CM} for more details).

For $\mu$-a.e. $x, y\in M$, we put
$$
[x, y]:= W^u(x)\cap W^s(y),
$$
which is either a singleton or an empty set.

\begin{definition}\label{defnsset} $ $
\begin{itemize}
\item[(1)]
A closed subset $\cR\subset M$ is called a \emph{hyperbolic product set} if for any $x, y\in \cR$ we have
$\emptyset\ne [x,y]\in \cR$.
\item[(2)] A closed region $\cU\subset M$ is called a \emph{solid (hyperbolic) rectangle} if it is bounded by two unstable manifolds and two stable manifolds.
We denote by $\cU(\cR)$ the \emph{solid hull} of a hyperbolic product set $\cR$, that is, the minimal solid rectangle that contains $\cR$.
\item[(3)]
Let $\cR$ be a hyperbolic product set, and $\sigma\in\{s, u\}$.
We call
\beq\label{def star collection}
\Gamma^\sigma(\cR):=\{W^\sigma(x)\cap \cU(\cR):\ \ x\in \cR\}
\eeq
the \emph{$\sigma$-collection} of $\cR$. Note that $\cR=\Gamma^u(\cR)\cap\Gamma^s(\cR)$, which is indeed a product set.
\item[(4)]
Let $\cR$ be a hyperbolic product set, and $\sigma\in\{s, u\}$.
A subset $\widetilde{\cR}\subset \cR$ is called a \emph{$\sigma$-subset} of $\cR$
if $W^\sigma(x)\cap \widetilde{\cR}=W^\sigma(x)\cap \cR$ for $\mu$-a.e. $x\in \widetilde{\cR}$.
Given a solid rectangle $\cU$, a \emph{solid $\sigma$-subrectangle} of $\cU$ can be defined in a similar fashion.
\end{itemize}
\end{definition}

We notice that the stable and unstable manifolds at each point of
a hyperbolic product set $\cR$
exist and cross beyond the boundaries of the solid rectangular hull $\cU(\cR)$.
In practice, the singular set $\cS_{\pm \infty}$ is usually dense in $M$, and
thus arbitrarily short stable and unstable manifolds exist almost everywhere.
In such case, hyperbolic product sets are nowhere dense. Nevertheless,
we can consider the $\mu$-measure of a hyperbolic product set.
The lemma below directly follows from
the absolute continuity assumption (\textbf{H3})(3).

\begin{lemma}\label{lem su-prod}
Given a hyperbolic product set $\cR$, let $\cR^s$ be an $s$-subsets of $\cR$
and $\cR^u$ be a $u$-subset of $\cR$.
Then $\mu(\cR^s\cap \cR^u)>0$ if and only if
$\mu(\cR^s)>0$ and $\mu(\cR^u)>0$.
\end{lemma}

\subsection{The Magnet -- A Special Hyperbolic Product Set}

In \cite{CM}, a \emph{magnet} is defined to be a special hyperbolic product set of positive measure that is used in the coupling lemma to attract the images of proper standard families,
which provides the foundation for the coupling algorithm. To construct such a special hyperbolic product set,
we first need to control the diameter of these hyperbolic sets such that
the oscillations of the SRB densities are small, and
the Jacobian of the stable holonomy map
takes values close to one.


By the bounded distortion of unstable Jacobian and
absolute continuity of the stable holonomy in Assumption (\textbf{H4}), we can prove the following properties.

\begin{lemma}\label{lem: abs cts}
There is $\delta_0>0$ such that
for any solid rectanlge $\cU(\cR)$ with diameter $\delta\in (0, \delta_0)$,
we have
\begin{itemize}
\item[(1)] for any unstable manifold $W\in\Gamma^u(\cR)$ and any $x, y\in W$,
\beqn
\left|\dfrac{\rho_W(x)}{\rho_W(y)}-1\right|<0.01,
\eeqn
where $\rho_W$ is the SRB density function given by \eqref{dens}.
\item[(2)]
for any two unstable manifolds $W^1, W^2\in\Gamma^u(\cR)$, we have
\beqn
\left|\dfrac{m_{W^2}(\bh^{s}(B))}{m_{W^1}(B)} - 1 \right|<0.01,
\eeqn
for any measurable subset $B\subset W^1_*$,
where $W^1_*$ is given by \eqref{abs cts domain}.
\end{itemize}
\end{lemma}

We are now ready to define the concept of magnets.

\begin{definition}
A hyperbolic product set $\cR$ is called a \emph{magnet} if $\mu(\cR)>0$,
and its solid rectangular hull $\cU(\cR)$ has diameter less than the constant $\delta_0$ given in
Lemma \ref{lem: abs cts}.
\end{definition}

Given a magnet $\cR$, we shall consider
the standard family induced by the SRB measure $\mu$
restricted on the $u$-collection $\Gamma^u:=\Gamma^u(\cR)$, that is,
\beq\label{def u family}
\cG(\cR):=(\Gamma^u, \mu|_{\Gamma^u})=\{(W, \mu_W), W\in \Gamma^u, \lambda^u\},
\eeq
with the factor measure $\lambda^u$.
It is clear that $\mu|_{\Gamma^u}(M)=\mu(\cup_{W\in \Gamma^u} W)\ge \mu(\cR)>0$.

\subsection{Construction of a Magnet $\hcR$}\label{sec: R star}

Special hyperbolic product sets  were constructed in \cite{CZ09} and \cite[\S 7.12]{CM}.
To make our later construction of $\cR^*$ complete, we sketch the construction of a magnet $\hcR$ in this subsection.

Choose $\delta_1\in (0, \delta_0)$ and $\eta\in (0, 1)$ such that:
$$
1-\eta^{\bs_0}(c_2+c_3 \delta_1^{\bs_0})>0, \ \ \text{and}\ \ c_4\delta_1^{\frac{1}{1-\bq_0}}< \delta_0,
$$
where the constant $\delta_0$ is given in Lemma \ref{lem: abs cts}, and
the constants $c_2, c_3$ and $c_4$ were given in Lemma \ref{lem: global}.
As a result,
$\mu_W(N_{\delta_2})>0$
for any standard pair $(W, \mu_W)$ with $|W|>\delta_1$,
where we set $\delta_2:=\eta \delta_1$.

Fix such a standard pair $(W_0, \mu_{W_0})$.
We choose a density point $\hx\in W_0\cap N_{\delta_2}$\footnote{
Here the point $\hx$ is chosen to be a $\mu$-density point in $M$, as well as
a $\mu_{W_0}$-density point in $W_0\cap N_{\delta_2}$.
}
and a small constant
\beq\label{proper constant 2}
0<\hdelta<0.1 \min\left\{\delta_2,\ c_4\delta_1^{1/(1-\bq_0)}\right\},
\eeq
such that $\mu_{\hW}(N_{\delta_2})>0.99$,
where $\hW$ is a closed sub-manifold of $W_0$
with center $\hx$ and of radius $\hdelta$.
We further assume that the endpoints $\partial \hW$ belong to $N_{\delta_2}$.
Set
\beqn\label{hatgammas}
 \widetilde\Gamma^s=\{ W^s(y):\  y\in \hW \cap N_{\delta_2}\}.
\eeqn
Note that $\widetilde\Gamma^s$ is the collection of all maximal stable manifolds
along $\hW \cap N_{\delta_2}$, which stick out both sides of $W_0$
by at least $c_4 \delta_1^{1/(1-\bq_0)}\ge 10\hdelta$.

We say that an unstable manifold $W$ fully cross $\widetilde\Gamma^s$ if
$W$ meets every stable manifold in $\widetilde\Gamma^s$.
Let $\widetilde\Gamma^u$ be the collection of all maximal unstable manifolds $W$ that fully cross
$\widetilde\Gamma^s$ such that
\beqn
d_{W^s(\hx)}(W\cap W^s(\hx),\ \hW\cap W^s(\hx))\le \hdelta.
\eeqn
Let
\beq\label{def R star}
\hcR= \widetilde\Gamma^s \cap \widetilde\Gamma^u:=
\{W^s\cap W^u:\ W^s\in \widetilde\Gamma^s \ \text{and}\  W^u\in \widetilde\Gamma^u\}
\eeq
denote the hyperbolic product set made by these two collections.
We can guarantee that $\mu(\hcR)>0$
for small value of $\hdelta$, by using Lemma \ref{lem: abs cts}
and the fact that $r^u(x)>0$ for $\mu$-a.e. $x\in M$.

By using Lemma \ref{lem: abs cts} and decreasing the value of $\hdelta$ if necessary,
we may assume that $\hcR$ is relatively dense (measure-theoretically)
in its solid rectangular hull $\hcU:=\cU(\hcR)$, that is,
\beq\label{hcR dense}
\mu(\hcR|\hcU)>0.99.
\eeq
Moreover, $|W\cap \hcU|< 3\hdelta$ for any stable/unstable manifold $W$ that fully crosses $\hcU$.
Therefore,
$\hcU$ is of diameter less than $6\hdelta<0.6\delta_0$, and thus
$\hcR$ is a magnet.
Also, we can assume that $\hcU\cap \cS_{\pm 1}=\emptyset$, and
\beq\label{no early intersection}
T^k\hcU\cap \hcU=\emptyset, \ \ \text{for} \ k=1, 2, \dots, \lfloor \log_{\Lambda} 100\rfloor+1.
\eeq

\subsection{The Coupling Lemma on Magnets}

In this subsection,
we review the {\it Coupling Lemma},
which was originally proved by Chernov and Dolgopyat (cf. \cite{C06, CD, D01}) for dispersing billiards,
see also \cite[\S 7.12--7.15]{CM}.
The coupling lemma was
generalized in \cite{CZ09} to systems under the assumptions (\textbf{H1})-(\textbf{H5})
for proper standard families, and
then in \cite{SYZ} to time-dependent billiards.

Let $\cR$ be a magnet
with its $s/u$-collections $\Gamma^{s/u}:=\Gamma^{s/u}(\cR)$.
We first recall the concept called the {\it{generalized standard family}}
with respect to the magnet $\cR$.
which was first introduced in \cite{VZ16}.

\begin{defn}\label{pseodugs}
Let $(\cW,\nu)$ be a standard family, such that $\cW\subset \Gamma^u$
is a measurable collection of unstable manifolds in $\Gamma^u$.
Then we define $(\cW,\nu)|_{\cR}:=(\cW|_{\cR}, \nu|_{\cR})$, where
$$
\cW|_{\cR}:=\{W\cap \cR:\ W\in \cW\}.
$$
For any integer $n\geq 0$, we call $T^{-n} ((\cW,\nu)|_{\cR})$ as a generalized standard family with index $n$.
\end{defn}

The Coupling Lemma in \cite{CD,CM} can then be restated in the language of
generalized standard families (see \cite{VZ16}).

\begin{lemma}\label{coupling1}
Let  $\cG^i=(\cW^i, \nu^i), i=1,2$,  be two  proper standard families in $M$.
There exist two sequences of generalized standard families $\{(\cW^i_n,\nu^i_n)\}_{n\ge 1}$
such that:
\beqn\label{decomposeGi}
\cG^i=\sum_{n=1}^{\infty}(\cW^i_n,\nu^i_n):=\left(\bigcup_{n=1}^{\infty}\cW^i_n,\sum_{n=1}^{\infty}\nu^i_n\right).
\eeqn
Moreover, we have the following properties:
\begin{itemize}
\item[(1)] \textbf{Proper return to $\cR$ at Step $n$}: Both $(\cW^1_n,\nu^1_n)$ and $(\cW^2_n,\nu^2_n)$ are  generalized standard families of index $n$;
\item[(2)] \textbf{Coupling $T^{n}_*\nu^1_n$ and $T^{n}_*\nu^2_n$ along stable manifolds in $\Gamma^s$}: For any measurable collection of stable manifolds $A\subset \Gamma^s$, we have
$
T^{n}_*\nu^1_n(A)=T^{n}_*\nu^2_n(A);
$
\item[(3)] \textbf{Exponential tail bound for uncoupled measures at Step $n$}:
We denote $\bar\nu_n^i:=\sum_{k\geq n}\nu^i_k$ as the uncoupled measures at the $n$-th step.
There exist constants $C_{\bc}>0$ and $\vartheta_\bc\in (0, 1)$, such that
\beqn\label{ctail2}
  \bar\nu_n^i(M)<C_{\bc}\vartheta_\bc^{n}.
\eeqn
\end{itemize}
\end{lemma}

Note that the above coupling lemma was only stated for proper standard families.
To deal with any standard family in $\fF_0$,
we may need sufficiently many iterates to make it proper.
Moreover, $\cW^i_m$ and $\cW^i_n$ may not be disjoint for $m\ne n$,
unless during the coupling process, one can couple the entire measure that properly return to $\cR$.

\begin{remark}
The coupling lemma stated in \cite{VZ16} begins with Step $n=0$,
with the coupling times being hitting times to the magnet $\cR$.
Our coupling lemma~\ref{coupling1} is concerned about return times, which is obtained by
performing the coupling lemma in \cite{VZ16} for the image families $T\cG^i, i=1,2$.
\end{remark}

\section{Proof of Theorem \ref{main}}\label{sec: main proof}

We shall prove Theorem \ref{main} in the following $4$ steps:
\begin{itemize}
\item[(i)]
Firstly, we obtain a decomposition of the magnet $\hcR$, by using
the coupling lemma on the particular standard family $\cG(\hcR)$ defined in \eqref{def u family}.
\item[(ii)] Secondly, we define an induced map $\hF:\hcR\to \hcR$, which is based on the coupling times,
and show that it admits a coding by countable symbols.
It follows that the periodic points of the original system $(M, T)$ exist and are dense in $\hcR$.
\item[(iii)] Thirdly, we shrink $\hcR$ to our final magnet $\cR^*$ whose solid boundaries are formed
by stable and unstable manifolds of a certain periodic orbit.
\item[(iv)] Lastly,, we apply the coupling lemma again on $\cR^*$
and prove that the first returns  to $\cR^*$ are indeed proper returns.
\end{itemize}

\subsection{Decomposition of the Magnet $\hcR$}\label{sec attraction}

Recall that $\cG(\hcR)$, given by \eqref{def u family},
is the proper standard family built on the $u$-collection $\hGamma^u:=\Gamma^u(\hcR)$
of the magnet $\hcR$ for the SRB measure $\mu$.
Note that $\hGamma^u$ consists of unstable manifolds
that fully cross the solid rectangle $\hcU:=\cU(\hcR)$ and terminate at its $s$-sides.
Also, we denote $\hGamma^s:=\Gamma^s(\hcR)$ the $s$-collection of $\hcR$.

For $i=1, 2$, we take $\cG^i$ as two copies of the family $\cG(\hcR)$,
and apply the coupling lemma~\ref{coupling1} with respect to the magnet $\hcR$.
Note that at each step $n\ge 1$,
we can actually couple the entire measure that properly return to $\hcR$.
Therefore,
there exists a sequence of generalized standard families $\{(\hcW_n,\hmu_n)\}_{n\ge 1}$ with the following properties:\begin{itemize}
\item[(1)] $(\hcW_n,\hmu_n)$ is a generalized standard family of index $n$;\footnote{
Some $\hcW_n$ may be null.
}
\item[(2)] $\{\hcW_n\}_{n\ge 1}$ are mutually disjoint, and
\beq\label{cWn union}
\cG(\hcR)=(\hGamma^u, \mu|_{\hGamma^u})=\sum_{n=1}^{\infty}(\hcW_n,\hmu_n) \pmod \mu;
\eeq
\item[(3)] Furthermore,
there exist $\widehat{C}>0$ and $\widehat{\vartheta}\in (0, 1)$ such that
\beqn\label{ctail1}
  \sum_{k\ge n} \hmu_k (M)<\widehat{C}\widehat{\vartheta}^{n}, \ \ \text{for any}\  n\ge 1.
\eeqn
\end{itemize}
In fact, we have
\beq\label{def Gammanu}
T^{n}(\hcW_n,\hmu_n)=(\hGamma_n^u, \mu|_{\hGamma_n^u})|_{\hcR},
\eeq
where $\hGamma_n^u$ is the sub-collections of $\hGamma^u$ containing
all components of $T^{n}W\cap \hcU$ that fully cross $\hcU$,
for some unstable manifold $W\in \hGamma^u$
which has not been attracted by the magnet $\hcR$ before the $n$-th iterate.
More precisely, $\Gamma_n^u$ is inductively defined as follows.
We set $\hGamma^u_1=\hGamma^u\cap T\hGamma^u$.
Once $\hGamma_k^u$ has been defined for $1\le k\le n$, we then define
\beq\label{def Gamma_n}
\hGamma^u_{n+1}=\hGamma^u\cap T\left[T^n\hGamma^u\cap
\left(\hGamma^u_n|_{\hcR} \cup T(\hGamma^u_{n-1}|_{\hcR})\cup \dots \cup T^n(\hGamma^u_1|_{\hcR}) \right)^c \right].
\eeq
We can further decompose $\hGamma^u_n$ as a disjoint union of
at most countably many sub-collections $\hGamma^u_{n, j}$,
each of which is maximal in the sense that
$T^{-n}$ is smooth on the solid rectangular hull
$\hcU^u_{n,j}$ that shadows $\hGamma^u_{n, j}$.
Such union is at most countable, since any two distinct hulls
must be separated by the singularity set $\cS_{-n}$,
which consists of at most countably many smooth curves (see \cite[\S 7.12]{CM}).
It is clear that $\hcU^u_{n,j}$ are mutually disjoint solid $u$-subrectangles of $\hcU$,
and $\hGamma^u_{n, j}=\hGamma^u|_{\hcU^u_{n,j}}$.

Next we set $\hcU^s_{n,j}:=T^{-n}(\hcU^u_{n,j})$. Then $\hcU^s_{n,j}$
are solid $s$-subrectangles of $\hcU$, due to the facts that
$T^{-n}$ is smooth on $\hcU^u_{n,j}$ and $\hcU\cap \cS_{\pm 1}=\emptyset$.
Accordingly, we decompose $\hcW_n$ as a disjoint union of
at most countably many generalized standard families
$
\hcW_{n,j}:=\hcW_n|_{\hcU^s_{n,j}}.
$
By our construction, we have that
$\hcW_{n, j}=\hGamma^u|_{\hcU^s_{n,j}\cap \ T^{-n}\hcR}$,
and $\hcW_{n,j}$ are mutually disjoint.
Then we define
\beq\label{Rsnj}
\hcR^s_{n,j}:=\hcW_{n,j}\cap\hcR
=\hGamma^u\cap\left((\hGamma^s\cap T^{-n}\hGamma^s)|_{\hcU^s_{n,j}}\right).
\eeq
It is obvious that $\hcR^s_{n,j}$ are mutually disjoint $s$-subsets of $\hcR$.
Moreover, by \eqref{cWn union},
$$
(\hGamma^u, \mu|_{\hGamma^u})
=\sum_{n\ge 1}(\hcW_n, \hmu_n)
=\sum_{n\ge 1}\sum_{j\ge 1} (\hcW_{n,j},  \hmu_{n,j}) \pmod \mu,
$$
where $\hmu_{n,j}=\hmu_n|_{\hcW_{n, j}}$.
Then we have $\hGamma^u\cap\hcR
=\bigcup_{n,j} \hcW_{n, j}\cap\hcR \pmod \mu$, that is,
\beq\label{R star s union}
\hcR=\bigcup_{n, j} \ \hcR_{n,j}^s \pmod \mu.
\eeq
Lastly we define
\beq\label{Runj}
\hcR^u_{n,j}:=T^n \hcR^s_{n,j}.
\eeq
It follows from \eqref{def Gammanu} and \eqref{Rsnj} that
$\hcR^u_{n,j}$ are mutually disjoint $u$-subsets of $\hcR$.
Moreover,
\beq\label{R star u union}
\hcR=\bigcup_{n,j}\ \hcR^u_{n,j} \pmod \mu.
\eeq
Indeed, by \eqref{R star s union},
$$
\sum_{n,j} \mu(\hcR^u_{n,j})=\sum_{n,j} \mu(T^n \hcR^s_{n,j})=
\sum_{n,j} \mu(\hcR^s_{n,j})=\mu(\hcR).
$$
To summarize, in this subsection, we have proved the following results.
\begin{proposition}\label{couplingdecompose}
For any magnet $\hcR$, there exist two unique decompositions $$\hcR=\bigcup_{n\geq 1}\bigcup_{j=1}^{N_n}\hcR^s_{n,j}=\bigcup_{n\geq 1}\bigcup_{j=1}^{N_n}\hcR^u_{n,j},$$
where for any $n\geq 1$, any $j=1,\cdots, N_n$, with $N_n\in \mathbb{N}\cup\{+\infty\}$, the set $\hcR^{s/u}_{n,j}$ is an $s/u$-subset of $\hcR$, and $T^n\hcR^s_{n,j}=\hcR^u_{n,j}$.
\end{proposition}

\subsection{Symbolic Coding and Periodic Points}\label{sec symbolic coding}

The magnet $\hcR$ can be coded by a countable full shift as follows.
Denote the alphabet set by
\beqn
\fA:=\{a=(n, j):\  \mu(\hcR^s_{a})>0 \}.
\eeqn
For any $x\in \hcR^s_a$ with $a=(n, j)\in \fA$, we set
\beqn
\tau(x)=\tau(a):=n.
\eeqn
The induced map $\hF: \hcR \to \hcR$ is given by
\beqn
\hF(x)=T^{\tau(x)}(x), \ \ \text{for any}\ x\in \hcR^s_a,
\eeqn
and it is undefined elsewhere. By \eqref{R star s union},
$\hF$ is well defined for $\mu$-a.e. $x\in \hcR$.

Note that $\hF$ is continuous on each $\hcR^s_a$, as it is the restriction of the smooth map
$T^{\tau(a)}|_{\hcU^s_a}$ on $\hcR^s_a$. Similarly, $\hF^{-1}$ is well defined and continuous on $\hcR^u_a$.
Moreover, $\hF$ preserves the SRB measure $\mu$, since for any measurable subset $B\subset \hcR$,
$$
\mu(\hF^{-1}(B))=\sum_{a\in \fA} \mu(\hF^{-1}(B)\cap \hcR^s_a)
=\sum_{a\in \fA}\mu(T^{-\tau(a)}(B\cap \hcR^u_a))
=\sum_{a\in \fA}\mu(B\cap \hcR^u_a)=\mu(B).
$$

Given a bi-infinite sequence $\ab=(a_k)_{k\in \IZ}\in \fA^\IZ$, we define
\beq\label{def pi}
\pi(\ab):=\bigcap_{k=-\infty}^\infty \hF^{-k} \hcR^s_{a_k}
=\lim_{m, \ell \to\infty} \bigcap_{k=-m}^\ell \ \hF^{-k} \hcR^s_{a_k} .
\eeq

\begin{remark}
Although $\hF^{-k}$ may not be defined on every point of $\hcR^s_a$ for $k\ne  -1$,
the intersection $\bigcap_{k=-m}^\ell \ \hF^{-k} \hcR^s_{a_k}$ in \eqref{def pi} should be understood as
$$
\hF^{-1}\left(\hcR^u_{a_0} \cap \hF^{-1}\left( \dots \cap \hF^{-1}\left(
\hcR^u_{a_{\ell-1}} \cap \hF^{-1} \hcR^u_{a_{\ell}}
\right)  \right) \right)
\bigcap \hF\left(\hcR_{a_{-1}}^s\cap \hF\left(\dots \cap
\hF\left(\hcR^s_{a_{-m+1}}\cap \hF\hcR^s_{a_{-m}} \right)   \right) \right).
$$
Then it is not hard to see that $\bigcap_{k=-m}^\ell \ \hF^{-k} \hcR^s_{a_k}$ is a closed subset of $\hcR^s_{a_0}$.
\end{remark}

\begin{lemma}\label{lem intersection}
For any $m, \ell\in \IN$ and any bi-infinite sequence $\ab=(a_k)_{k\in \IZ}\in \fA^\IZ$, we have
\beq\label{eq intersection}
\mu\left( \bigcap_{k=-m}^\ell  \hF^{-k} \hcR^s_{a_k}\right)>0.
\eeq
Consequently, $\pi(\ab)$ is well defined, and in fact, a single point in $\hcR^s_{a_0}$.
\end{lemma}
\begin{proof} It suffices to show \eqref{eq intersection} for the case when $m=0$, since
$$
\mu\left(\bigcap_{k=-m}^\ell \hF^{-k} \hcR^s_{a_k}\right)=
\mu\left(\hF^{m} \left(\bigcap^{m+\ell}_{k=0} \hF^{-k} \hcR^s_{a_{k-m}}\right)\right)=
\mu\left(\bigcap^{m+\ell}_{k=0} \hF^{-k} \hcR^s_{a_{k-m}}
\right).
$$
We then set $m=0$, and make induction on $\ell$. For $\ell=0$, it is obvious that $\mu(\hcR^s_{a_0})>0$.
Suppose now \eqref{eq intersection} holds for $\ell$, then by Lemma~\ref{lem su-prod},
$$
\mu\left( \bigcap_{k=0}^{\ell+1} \hF^{-k} \hcR^s_{a_k}\right)=
\mu\left( \left( \bigcap^{\ell}_{k=0} \hF^{-k} \hcR^s_{a_{k+1}}\right)\bigcap \hcR^u_{a_0}
\right)>0,
$$
as $\hcR^u_{a_0}$ is a u-subset of $\hcR$ with positive measure, and
$\bigcap_{k=0}^{\ell} \hF^{-k} \hcR^s_{a_{k+1}}$ is a s-subset of $\hcR$ with positive measure by induction assumption.

Consequently, any finite intersection
$
 \bigcap_{k=-m}^\ell \hF^{-k} \hcR^s_{a_k}
$
is a non-empty closed subset of $\hcR$.
Moreover, by the uniform hyperbolicity of $\hF=T^\tau$, the diameter of
$
 \bigcap_{k=-m}^\ell \hF^{-k} \hcR^s_{a_k}
$
decreases exponentially in the rate of order at least $\Lambda^{-\min\{m, \ell\}}$.
Therefore, by \eqref{def pi} and the compactness of $\hcR^s_{a_0}$,
$\pi(\ab)$ is a single point in $\hcR^s_{a_0}$.
\end{proof}

Lemma~\ref{lem intersection} directly implies that
the two-sided full shift $\sigma: \fA^\IZ\to \fA^\IZ$
is topologically conjugate to the induced map $\hF: \hcR_{\hF}\to \hcR_{\hF}$
via the map $\pi$,
where
$
\hcR_{\hF}:=\bigcap_{k=-\infty}^\infty \hF^{-k}\left( \bigcup_{a\in \fA} \hcR^s_a\right)
$
is the maximal $\hF$-invariant subset of $\hcR$ with full measure.
As an immediate result, we have

\begin{lemma}\label{lem period}
For any periodic sequence
$$
\ab=\overline{(a_0, a_1, \dots, a_{m-1})}:=\{(\ab_k):\ \ab_{k+jm}=a_k, \
\text{for any}\ 0\le k<m, \ j\in \IZ\}\in \fA^\IZ,
$$
the point $\pi(\ab)$ is an $m$-periodic point of $(\hcR_{\hF}, \hF)$,
and thus a periodic point  of $(M, T)$ with period $(\tau(a_0)+\dots+\tau(a_{m-1}))$.
\end{lemma}

It follows that periodic points of the original system $(M, T)$ is dense in $\hcR$. Indeed,
for any $\eps>0$ and $x\in \hcR$, there is $y=\pi(\ab)\in \hcR_{\hF}$, for some $\ab=(a_k)_{k\in \IZ}\in \fA^\IZ$,
that is $\eps$-close to $x$. Then we can choose a large $m\in \IN$ such that the periodic point
$y'=\pi(\ab')$ is $\eps$-close to $y$ and hence $2\eps$-close to $x$,
where $\ab'=\sigma^{m}\overline{(a_{-m}, \dots, a_{-1}, a_0, a_1, \dots, a_{m})}$.

\subsection{Construction of the Final Magnet $\cR^*$}

In this subsection, we construct the final magnet $\cR^*$ in Theorem \ref{main},
which is bounded by stable and unstable manifolds of a particular periodic orbit. Moreover, $R^*$ is special in the following sense.

\begin{definition}\label{def regular property}
We say a solid rectangle $\cU$ is {\it{perfect}} if no forward iterations of its stable boundary ever enter the interior of $\cU$, and no backward iterations of unstable boundary ever enter the interior of $\cU$. We say the hyperbolic product set $\cR$ is {\it{perfect}} if $\cU(\cR)$ is perfect.
\end{definition}

\begin{lemma}\label{perfectcU} There is a perfect rectangle $\cU^*$ is the rectangle $\hat\cU$.
\end{lemma}

To prove this lemma, we first prove a few lemmas. We choose two special symbols in $\fA$ as follows.
For any $a\in \fA$, the solid rectangle $\hcU^s_a=\cU(\hcR^s_a)$ is a solid $s$-subrectangle of $\hcU=\cU(\hcR)$.
By Lemma \ref{lem: abs cts},
there is $\delta_{a}>0$ such that
\beq\label{def delta_a}
0.99\delta_{a}\le |W\cap \hcU^s_{a}|\le 1.01\delta_{a}
\eeq
for any unstable manifold $W$ that fully cross $\hcU^s_{a}$.
We fix the number $\delta_a$ for each $a\in \fA$, and call it the (rough) $u$-diameter of $\hcU^s_a$.

Given $\eta\in (0,1)$, we denote by $\hcU^s(\eta)$ the solid $s$-subrectangle of $\hcU$
such that each $s$-side of $\hcU^s(\eta)$ is away from the nearest $s$-side of $\hcU$ by
$\eta$ multiple of the $u$-diameter of $\hcU$.

\begin{lemma}\label{lem u diameter}
There exist two symbols $a_0, a_1\in \fA$ such that
the $u$-diameters of $\hcU^s_{a_0}$ and $\hcU^s_{a_1}$ are comparable in the sense that
$\frac12\delta_{a_0}\le \delta_{a_1}\le 2\delta_{a_0}$.
Furthermore, $\hcU^s_{a_0}$ and $\hcU^s_{a_1}$ are almost centric in the sense that
they are both contained in $\hcU^s(0.3)$.
\end{lemma}

\begin{proof} By \eqref{no early intersection}, every solid $s$-subrectangle $\hcU^s_a$
occupies at most $\frac{1}{100}$ portion of $\hcU$ with respect to the SRB measure $\mu$.
If there were no comparable symbols whose solid $s$-rectangles intersecting $\hcU^s(0.4)$, then
$\hcR\cap \hcU^s(0.4)\subset \bigcup_{a\in \fA} \hcU^s_a\cap \hcU^s(0.4)$
occupies at most
$$
1.01\times \frac{1}{100}\left(1+\frac12+ \frac{1}{2^2}+\dots\right)<0.03
$$
portion of $\hcU$. However, by \eqref{hcR dense}, $\hcR\cap \hcU_{0.4}^s$ should at least take up
$(0.2-0.01)\times 0.99>0.1$ portion of $\hcU$, which is a contradiction.
Thus, there exist two comparable symbols $a_0$ and $a_1$ such that both
$\hcU^s_{a_0}$ and $\hcU^s_{a_1}$ intersect $\hcU^s(0.4)$,
and hence are both contained in $\hcU^s(0.3)$.
\end{proof}

\begin{remark} By Lemma \ref{lem: abs cts} and the fact that $\hF$ preserves $\mu$,
one can also use $\delta_a$ as the (rough) $s$-diameter of $\hcU^u_a=\cU(\hcR^u_a)$,
by weakening the factors in \eqref{def delta_a} to $0.9\sim 1.1$.
\end{remark}

Next, we fix two symbols $a_0, a_1\in \fA$ given by Lemma \ref{lem u diameter},
and define a periodic sequence
\beq\label{def ab}
\ab=
\overline{(\underbrace{a_0, a_0, \dots, a_0}_{N_0}, a_1, \underbrace{a_0, a_0, \dots, a_0}_{N_0})}\in \fA^\IZ,
\eeq
for a sufficiently large $N_0\in \IN$.
By Lemma \ref{lem period},
$x_0=\pi(\ab)$ is a periodic point of $(\hcR, \hF)$ with period $(2N_0+1)$,
and thus a periodic point of $(M, T)$ with period $(2N_0\tau_0+\tau_1).$\footnote{
$(2N_0\tau_0+\tau_1)$ need not be the least period.
}

Note that
the points $x_0$ and $\hF(x_0)=\pi(\sigma(\ab))\ne x_0$ both belong to
$\hcR^s_{a_0}\cap \hcR^u_{a_0}\subset \hcR$.
We denote the solid rectangles
$\hcU_{a_0}^s:=\cU(\hcR^s_{a_0})$,
$\hcU_{a_0}^u:=\cU(\hcR^u_{a_0})$, and
$\hcU_{a_0}:=\hcU_{a_0}^s\cap \hcU_{a_0}^u$,
then the points $x_0$ and $\hF(x_0)$ both belong to $\hcU_{a_0}$.

We denote by $\cO(x_0)$ the $T$-orbit of the periodic point $x_0$ in $M$.

\begin{lemma}\label{lem fully cross}
If $x\in \cO(x_0)\cap \hcU_{a_0}$, then $W^u(x)$ fully crosses $\hcU^s_{a_0}$ and
$W^s(x)$ fully crosses $\hcU^u_{a_0}$.
\end{lemma}

\begin{proof}
We only prove this lemma for the unstable manifolds. The case for the stable manifolds would be similar.

Let $\delta_{a_0}$ be the $u$-diameter of $\hcU^s_{a_0}$,
and recall that  $r^u(x)$ the distance from $x$ to the nearest endpoint of $W^u(x)$.
By Lemma \ref{lem: abs cts}, an unstable manifold $W^u(x)$ must fully cross $\hcU^s_{a_0}$
if $x\in \hcU_{a_0}^s$ and $r^u(x)>1.01\delta_{a_0}$.

Note that the $\hF$-orbit of $x_0$ are given by
$$
\hF^k(x_0)=
\begin{cases}
T^{k\tau_0}(x_0), \ & \ 0\le k\le N_0, \\
T^{(k-1)\tau_0+\tau_1}(x_0), \ & \ N_0<k\le 2N_0,
\end{cases}
$$
which belong to $\hcR^s_{a_0}$
except that $\hF^{N_0}(x_0)\in \hcR^s_{a_1}$.
Either way, $\hF^k(x_0)$ all belong to $\hcR$,
and thus their unstable manifolds fully cross $\hcU^s_{a_0}$.
Moreover, by the almost centric property in Lemma~\ref{lem u diameter},
we have that $r^u(\hF^k(x_0))\ge 0.3 \delta_{a_0}$ for $k=0, \dots, N_0-1, N_0+1, \dots, 2N_0$,
and $r^u(\hF^{N_0}(x_0))\ge 0.3 \delta_{a_1}$.
There are three cases for other points $x=T^m(x_0)\in \cO(x_0)\cap \hcU_{a_0}$:
\begin{itemize}
\item[(i)] If $k\tau_0< m< (k+1)\tau_0$ for $k=0, 1, \dots, N_0-1$: in this case, we actually have
$m\ge k\tau_0+\lfloor \log_{\Lambda} 100\rfloor+1$ by \eqref{no early intersection}.
Since $T^{k\tau_0}(x_0)=\hF^k(x_0)\in \hcR^s_{a_0}$
and $T^{\tau_0}$ is smooth on $\hcU^s_{a_0}=\cU(\hcR^s_{a_0})$,
$r^u(x)$ is no less than $100\ r^u(\hF^k(x_0))\ge 30\delta_{a_0}$.
\item[(ii)] If $N_0\tau_0<m<N_0\tau_0+\tau_1$:
we have $m\ge N_0\tau_0+\lfloor \log_{\Lambda} 100\rfloor+1$ by \eqref{no early intersection}.
Since $T^{N_0\tau_0}(x_0)=\hF^{N_0}(x_0)\in \hcR^s_{a_1}$
and $T^{\tau_1}$ is smooth on $\hcU^s_{a_1}=\cU(\hcR^s_{a_1})$,
$r^u(x)$ is no less than $100\ r^u(\hF^{N_0}(x_0))\ge 30\delta_{a_1}\ge 15\delta_{a_0}$ by
the comparable property in Lemma~\ref{lem u diameter}.
\item[(iii)] $(N_0+k)\tau_0+\tau_1<m<(N_0+k+1)\tau_0+\tau_1$ for $k=0, 1, \dots, N_0-1$: this case
is similar to Case (i).
\end{itemize}
In all cases, $W^u(x)$ fully crosses $\hcU^s_{a_0}$.
\end{proof}

At last, we let
$$
x'_0:=[x_0, \hF(x_0)]=W^u(x_0)\cap W^s(\hF(x_0)), \ \ x''_0:=[\hF(x_0), x_0]=W^u(\hF(x_0))\cap W^s(x_0),
$$
and we denote by $\cU_0^*$ the solid rectangle with four corner points $\{x_0, x'_0, \hF(x_0), x''_0\}$
in the counter-clockwise direction.
Note that $\cU_0^*\subset \hcU_{a_0}$.
Let $\delta_0^*$ be the diameter of $\cU_0^*$, and $d^*_0$ be the minimum of
the distances of $\partial \cU_0^*$ and $\partial \hcU_{a_0}$ in $s$- and $u$-direction.
By taking $N_0$ in \eqref{def ab} sufficiently large, we may assume that $\delta_0^*<(\Lambda-1) d^*_0$.

We then set
\beqn
\cD^s:=\{W^s(x)\cap \cU_0^*: \ x\in \cO(x_0)\cap\hcU_{a_0}\}, \ \text{and}\
\cD^u:=\{W^u(x)\cap \cU_0^*: \ x\in \cO(x_0)\cap\hcU_{a_0}\},
\eeqn
that is, $\cD^{s/u}$ are the families of intersection curves of
the stable/unstable manifolds of periodic points in $\cO(x_0)\cap\hcU_{a_0}$ within $\cU_0^*$.
Note that both families $\cD^s$ and $\cD^u$ are finite (could be empty).
By Lemma~\ref{lem fully cross},
each $W\in \cD^{s/u}$ would fully cross $\cU_0^*$.

\begin{lemma}\label{lem partition}
The family $\cD^s$ is $T$-invariant in the following sense:
for any $k\in \IN$ and any $W\in \cD^s$,
$T^k W\cap \cU_0^*$ is either empty, or it is
a sub-curve of some  $W'\in \cD^s$.
Similarly, the family $\cD^u$ is $T^{-1}$-invariant.
\end{lemma}

\begin{proof}
We only prove for the $T$-invariance of the family $\cD^s$, and
then $T^{-1}$-invariance of $\cD^u$ can be proved in a similar fashion.

For any curve $W\in \cD^s$,
then we can extend $W$ in one side until it hits a periodic point $x_1\in \cO(x_0)\cap\hcU_{a_0}$.
We denote the extended curve by $W_1$,
then the length $|W_1|\le \delta_0^*+d_0^*$.
For any $k\in \IN$, we have that $|T^kW_1|\le \Lambda^{-1}(\delta_0^*+d_0^*)<d_0^*$.
Therefore, if $\emptyset\ne T^k W\cap \cU_0^*\subset T^k W_1\cap \cU_0^*$,
then the periodic point $x_2=T^k(x_1)\in \cO(x_0)$, as one endpoint of $T^kW_1$,
has distance to $\cU_0^*$ less than $d_0^*$. Thus $x_2\in \hcU_{a_0}$,
and $T^k W\cap \cU_0^*$ is a sub-curve of $W^s(x_2)\cap \cU_0^*$.
\end{proof}

The families $\cD^{s/u}$ divide the solid rectangle $\cU_0^*$ into finitely many solid sub-rectangles.
We choose a such solid sub-rectangle, and denote by $\cU^*$. By Lemma~\ref{lem partition},
$\cU^*$ is perfect, as for any $k\in \IN$,
the $T^k$-images of $s$-sides of $\cU^*$, as well as $T^{-k}$-images of $u$-sides of $\cU^*$,
do not intersect $\mathrm{int}(\cU^*)$.

In our construction, it is not hard to see that
there are a positive measure of stable manifolds surrounding the $s$-sides of $\cU^*$,
as well as a positive measure of unstable manifolds surrounding the $u$-sides of $\cU^*$,
both of which fully cross $\cU^*$.
This finished the proof of Lemma \ref{perfectcU}.\qed

Now we will add the $2$ stable manifolds of the boundary components of $\cU^*$ into the singular set $\cS_1$, as well as the $C^{\infty}$ extension of these $2$ curves into $\cS_1$, such that each one of them ends at another singular curve in $\cS_0$. Similarly, we will add the $2$ unstable manifolds of the boundary components of $\cU^*$ into the singular set $\cS_{-1}$, as well as the $C^{\infty}$ extension of these $2$ curves into $\cS_{-1}$, such that each one of them ends at another singular curve in $\cS_0$. Now we define the new singular sets as $\cS_{\pm 1}^*$. One can check that $T: M\setminus\cS_1^*\to M\setminus \cS_{-1}^*$ satisfies $(\textbf{H2})$. It follows from \cite{KS86} that condition (\ref{def q0}) implies that there exist stable manifold  and unstable manifold
for $m$-a.e. $x\in M$.
For any $n\geq 1$, let $\ds\cS^*_{\pm n}=\cup_{m=0}^{n-1} T^{\mp m}\cS^*_{\pm 1}$
be the modified singularity set of $T^{\pm n}$,
and $\ds\cS^*_{\pm \infty}=\cup_{m\geq 0} \cS^*_{\pm m}$.
Note that any new maximal stable/unstable manifold $W^{s/u}$ is an open connected curve in $M\backslash \cS^*_{\pm \infty}$,
usually with two endpoints in $\cS^*_{\pm \infty}$.
\begin{proposition}\label{exist regular magnet}
There is a perfect magnet $\cR^*$ in the magnet $\hcR$.
\end{proposition}
 \begin{proof}
 Using results in Lemma \ref{perfectcU}, we define our final magnet $\cR^*$ by  thickening $\cU^*\cap \hcR$ in the following way: we consider  all  stable  manifolds that fully cross the unstable boundary of $\cU^*$, and denote the collection of their intersections with $\cU^*$  as  $\Gamma^s$; similarly, we consider all unstable manifolds that fully cross the stable boundary of $\cU^*$, and denote the collection of their intersection with $\cU^*$
 as $\Gamma^u$. Then we define
$$\cR^*=\Gamma^s\cap \Gamma^s$$
It is clear that $\cR^*$ is perfect as its solid rectangular hull is exactly $\cU^*$.

 Thus we have complete the proof of the proposition \ref{exist regular magnet}.

\end{proof}

\subsection{The First Return to $\cR^*$ is Proper}
In this subsection, we will prove the perfect magnet $\cR^*$ is our final objective magnet in Theorem \ref{main}.

By Proposition \ref{couplingdecompose}, we know that the hyperbolic set $\cR^*$ also has two unique decompositions \beq\label{firstdecom}\cR^*=\bigcup_{n\geq 1}\bigcup_{j=1}^{N_n}\cR^s_{n,j}=\bigcup_{n\geq 1}\bigcup_{j=1}^{N_n}\cR^u_{n,j},\eeq
where for any $n\geq 1$, any $j=1,\cdots, N_n$, with $N_n\in \mathbb{N}\cup\{+\infty\}$, the set $\cR^{s/u}_{n,j}$ is a $s/u$-subrectangle of $\cR^*$, and $T^n\cR^s_{n,j}=\cR^u_{n,j}$.

We recall that $\cU^{s/u}_{n,j}$ is defined to be the minimal solid ($s/u-$)rectangle containing $\cR^{s/u}_{n,j}$. Clearly, we also have $T^n\cU^s_{n,j}=\cU^u_{n,j}$. Now we define two subsets $\cU^{s/u}\subset \cU^*$ in the following:
$$\cU^s:=\bigcup_{n\geq 1}\bigcup_{j=1}^{N_n}\cU^s_{n,j},\,\,\,\,\,\,\,\cU^u:=\bigcup_{n\geq 1}\bigcup_{j=1}^{N_n}\cU^u_{n,j}$$

Note that $\cU^*\setminus \cU^s$ contains at most countably many connected, solid $s$-subsets, denoted the set of them by $\mbox{Gap}_s^{(1)}=\{G_m^1, m\geq 1\}$, which we call {\it{s-gaps of type I}}.
In addition, $\cU^s\setminus \Gamma^s$ also contains countably many connected, solid $s$-subsets. We call them {\it{ s-gaps of type II}}, and denote the set of them by
$\mbox{Gap}_s^{(2)}=\{G_m^2, m\geq 1\}$.
We call each element $G\in\mbox{Gap}_s:=\mbox{Gap}_s^{(1)}\cup\mbox{Gap}_s^{(2)}$ a s-gap of $\cU^*$. Let $\cG^1=\cup_{m\geq 1} G^1_m$ and  $\cG^2=\cup_{m\geq 1} G^2_m$, the by the definitions we have $$\cG^2\subset\cU^s,\ \ \ \  \cU^*=\cG^1\cup\cU^s.$$

Similarly, by time reversibility of the map, we can define the u-Gap $G\in \mbox{Gap}_u:=\mbox{Gap}_u^{(1)}\cup\mbox{Gap}_u^{(2)}$.
We call $U$ is a solid rectangle in $\cU^*$, if $U$ is bounded by a pair of stable manifolds and a pair of unstable manifold. The following lemma describes the dynamical property of the gaps.

\begin{lemma}\label{gap2} For any s-gap $G$ of $\cU^*$,  if there exist a $s$-rectangle $\hG\subset G$, $k\geq 1$, such that $T^k \hG$ is a solid rectangle in the interior of $\cU^*$, then there exists a s-gap $G'$ of $\cU^*$,  such that  $T^k \hG \subset G'$. i.e., any solid s-rectangle in a gap  will only return to another s-gap of $\cU^*$ if it returns to the interior of $\cU^*$ as a solid rectangle.
\end{lemma}
\begin{proof}
We first consider s-gap of type I. For any $G\in \mbox{Gap}_s^{(1)}$, assume there exist a s-subrectangle $\hG\subset G$,  $k\geq 1$, such that $T^k \hG$ intersects $\cU^*$ nontrivially at a solid rectangle $U$. Assume $U$ does not belong to any gap, then $U$ can be extended horizontally to a $u$-subrectangle of $\cU^*$. By the definition of the gap, we know that for almost all $x\in G$, the stable manifold $W^s(x)$ could not stretch to both unstable boundary of the solid s-rectangle $G$. Note that $T^k$ is continuous on the solid rectangle $\hG$, thus $T^k \hG$ is a solid rectangle containing points with short stable manifolds that could not stretch to both sides of $T^k \hG$. This implies that $\Gamma^s\cap T^k \hG$ is an empty set. Thus $T^k \hG$ is contained in a s-gap of $\cU^*$.

Next we consider s-gaps of type II.
By the definition of $\mbox{Gap}_s^{(2)}$, we know that for each $G\in \mbox{Gap}_s^{(2)}$, there is a solid  $s$-rectangle $\cU_{n,j}^s$, such that $G\subset \cU_{n,j}^s$. Note that $T^n$ is continuous on the solid $s$-rectangle $\cU_{n,j}^s$, and $\Gamma^s\cap \cU_{n,j}^s$ is mapped to the set of all the stable manifolds in $T^n \cU_{n,j}^s$ which are fully stretched in  $T^n \cU_{n,j}^s$ in the stable direction. Here we call the countable regions in $T^n\cU_{n,j}^s\setminus T^n\Gamma^s$ {\it{the s-gaps of }} $T^n \cU_{n,j}^s$. One can check that any s-gap of $T^n \cU_{n,j}^s$ only contains point with  short stable manifold that could not fully stretch in $T^n\cU_{n,j}^s$. Thus it can not contain any stable manifolds in $\Gamma^s$. This implies that it must belong to some s-gap $G'$ of  $\cU^*$. Note that $T^n G'$ is a gap of $T^n \cU_{n,j}^s$, thus it must belongs to a s-gap of $\cU^*$.
\end{proof}

\begin{lemma}\label{Unn}
 $T^n\cR^s_n$ returns to $\cR^*$ for the first time: $T^k\cR^s_n\cap \cR^*$ is either empty or trivial set (with respect to measure $\mu$), for all $n\geq 1$ and $k=1,\cdots, n-1$.
\end{lemma}
\begin{proof}
 The decomposition of $\cR^{*}$ was given by (\ref{firstdecom}). It is enough to show that the first return of $\cR^s_{n,j}$ is indeed already a proper return, for any $n\geq 1$, $j=1,\cdots, N_n$. We prove this fact by induction on $n$.

 For $n=1$, the statement is clearly true.

 Assume the statement holds for $\cR_{m,j}^s$, $m=1, \dots, n-1$, $j\geq 1$,
 then we will consider $\cR_{n,j}^s$.
 We prove this fact by way of contradiction. Suppose there exists $1<k<n$, such that $T^k \cR^s_{n,j}$ intersects $\cR^*$ nontrivially. This also implies that  $T^k \cU^s_{n,j}$ intersects $\cR^*$ nontrivially.  Suppose
  $T^k\cU^s_{n,j}$ intersects the unstable boundary of $\cR^*$, say $W^u_1$, see Figure \ref{fig1}. Note that now the backward image of $W^u_1$ returns to the interior of $\cU^*$ under $T^{-k}$, which contradict to the fact that $\cU^*$ is perfect.

 \begin{figure}[h]
\centering \psfrag{T1}{\scriptsize$T^k$} \psfrag{T2}{\scriptsize$T^{n-k}$}
 \psfrag{TkW}{\scriptsize$T^{-k}W^u_1$}
 \psfrag{u}{\scriptsize$\cU^s_{n,j}$}
\includegraphics[width=5in]{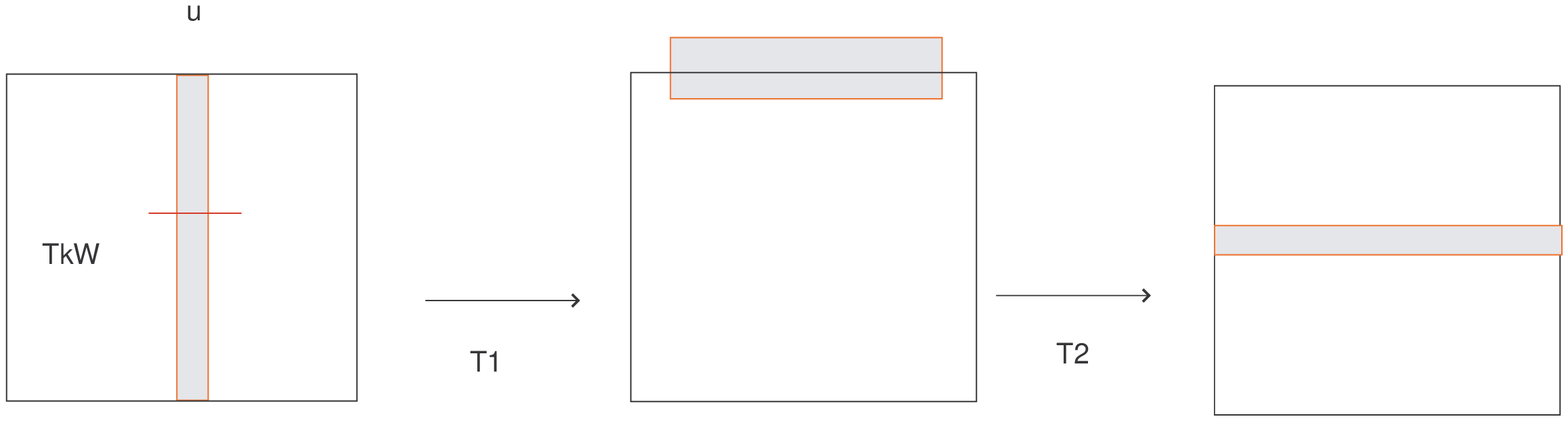}
\renewcommand{\figurename}{Fig.}
\caption{\small\small\small\small\small{ $T^k \cU^s_{n,j}$ intersects $W^u_1$.\label{fig1}}}
\end{figure}

Similarly, suppose
$T^k\cU^s_{n,j}$ intersects the stable boundary of $\cR^*$, say $W^s_1$, see Figure \ref{fig2}. Note that now the forward image of $W^s_1$ returns to the interior of $\cU^*$ under $T^{n-k}$, which contradict to the fact that $\cU^*$ is perfect.

 \begin{figure}[h]
\centering \psfrag{T1}{\scriptsize$T^k$} \psfrag{T2}{\scriptsize$T^{n-k}$}
 \psfrag{TkW}{\scriptsize$T^{n-k}W^s_1$}
 \psfrag{u}{\scriptsize$\cU^s_{n,j}$}
\includegraphics[width=5in]{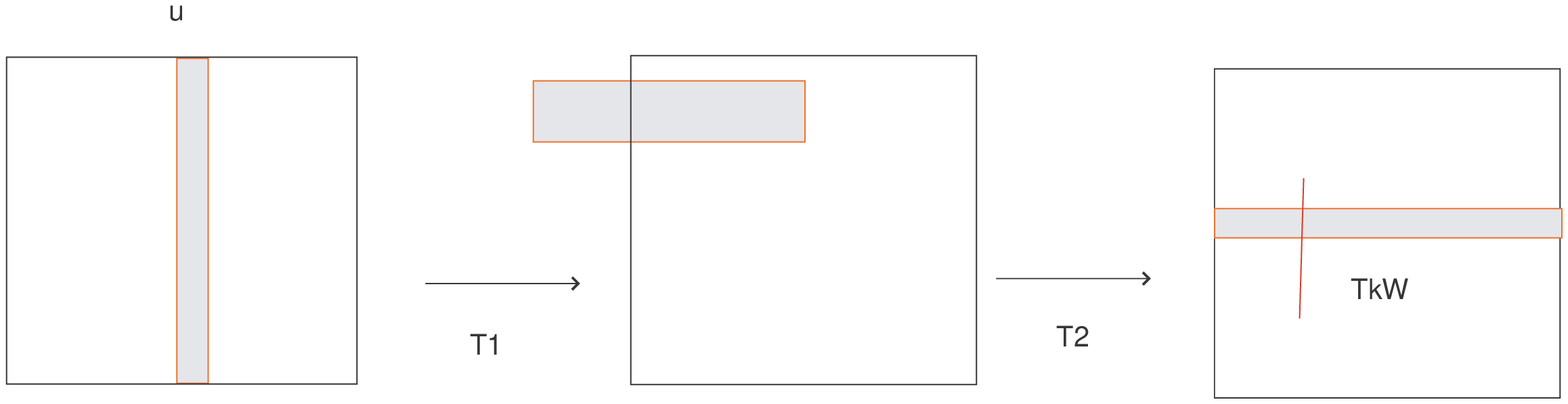}
\renewcommand{\figurename}{Fig.}
\caption{\small\small\small\small\small{ $T^k \cU^s_{n,j}$ intersects $W^s_1$.\label{fig2}}}
\end{figure}

Finally, we suppose
$T^k\cU^s_{n,j}$ lies in the interior of $\cU^*$, see Figure \ref{fig3}.
 Then there are  two cases to consider.\\

\noindent \textbf{Case I.} We assume ${\cU}_{n,j}$ does not belong to any other $s$-rectangle $\cU_{m,i}$, with $m<n$, $i\geq 1$.

 Using the aligned property of singular curves,  we know that $T^k\cU^s_{n,j}$ can be extended to a $u$-rectangle in $\cU^*$ that properly cross $\cU^*$. We denote the preimage of this extended $u$-solid rectangle under $T^{-k}$ by $\hat{\cU}_{n,j}$. We  claim that $\hat{\cU}_{n,j}$ is a $s$-subrectangle that does not intersect with any $\cR_{m,i}^s$, for $m=1,\cdots, n-1$.
Indeed if $\hat{\cU}_{n,j}$ does not intersect any other s-subset of index $(m,i)$, such that $m< n$, then we know that $T^k \hat{\cU}_{n,j}$ properly returns to $\cU^*$ for the first time, thus  $k=n$, which is a contradiction.

 \begin{figure}[h]
\centering \psfrag{T1}{\scriptsize$T^k$} \psfrag{T2}{\scriptsize$T^{n-k}$}
 \psfrag{TkW}{\scriptsize$T^{n-k}W^s_1$}
 \psfrag{u}{\scriptsize$\cU^s_{n,j}$}
\includegraphics[width=5in]{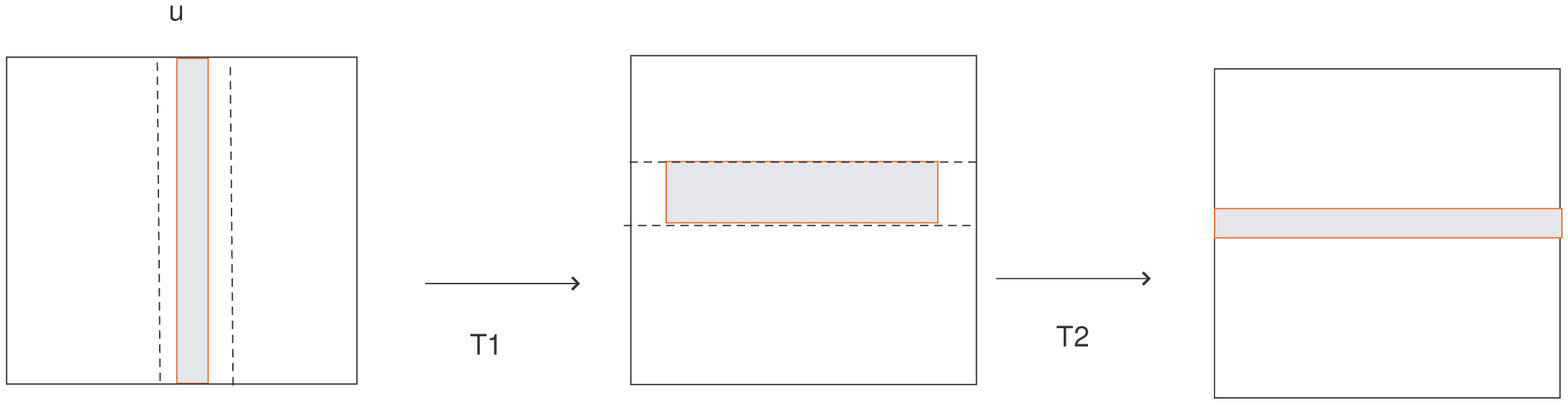}
\renewcommand{\figurename}{Fig.}
\caption{\small\small\small\small\small{ $T^k \cU^s_{n,j}$ lies in the interior of $\cU^*$.\label{fig3}}}
\end{figure}

Suppose that $\hat{\cU}_{n,j}$ does  intersect a $s$-rectangle of index $(m,i)$, with smaller index $m<n$, then one $s$-boundary of $\cU^*$ that intersects with $T^m \cU_{m,i}^s$ will return to the interior of $\cU^*$ after another $n-m$ iterations, which contradicts the definition of the perfect set $\cU^*$. This verifies the claim.\\

\noindent \textbf{Case II.} We assume ${\cU}^s_{n,j}$ entirely belongs to a $s$-rectangle $\cU_{m,i}^s$, with $m<n$, for some $i\geq 1$.  This implies that $T^m \cU_{n,j}^s$ lies in a s-gap $G$ of $\cU^*$, according to the construction of $\cU^s_{n,j}$ in Proposition \ref{couplingdecompose}.
We further decompose  $T^k\cU^s_{n,j}=A\cup B$ into two sets, such that $A$ belongs to the $Gap_s$, the other set $B$ does not. Then $B$ is a disjoint union of rectangles in the complement of the s-gaps. This implies that each rectangle can be extended in the stable direction to a s-rectangle in $\cU^*$. More precisely, $B$ can be extended to a s-rectangle  $\cU^*_{n-k,l}$, for some $l\geq 1$. This implies that $T^{n-k} B \subset T^{n-k}\cU^*_{n-k,l}$. By time reversibility, this implies that $B$ lies inside the u-gaps of $\cU^*$. Thus we have shown that $A$ lies in the s-gaps, and B lies in the u-gaps. Thus
$T^k\cU^s_{n,j}\cap \cR^*$ is empty, which is a contradiction.

 \begin{figure}[h]
\centering \psfrag{T1}{\scriptsize$T^m$} \psfrag{T2}{\scriptsize$T^{n-k}$}
\psfrag{T3}{\scriptsize$T^{k-m}$}
 \psfrag{G}{\scriptsize$G$}
  \psfrag{V}{\scriptsize$\cU^s_{n-k,l}$}
   \psfrag{TW}{\scriptsize$T^{k-m}W$}
 \psfrag{u}{\scriptsize$\cU^s_{n,j}$}
 \psfrag{U}{\scriptsize$\tilde{U}$}
\includegraphics[width=6in]{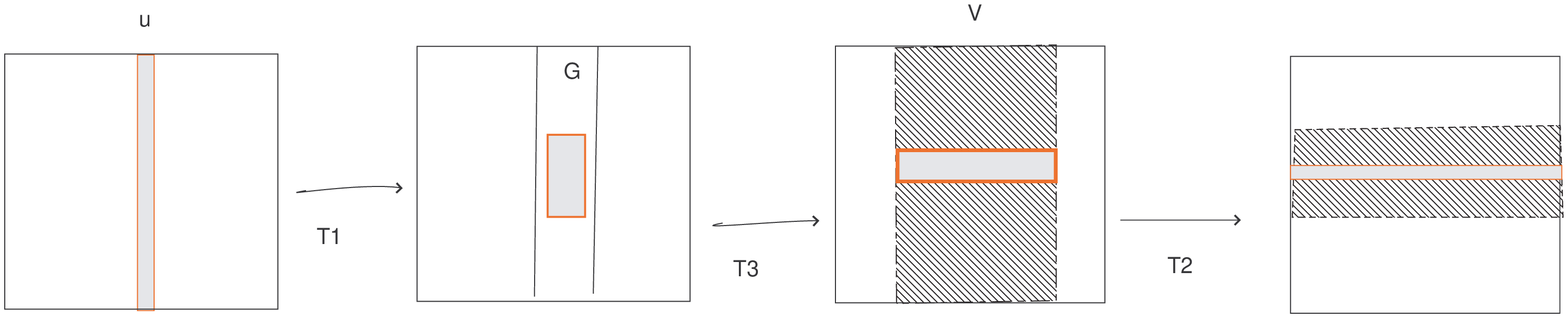}
\renewcommand{\figurename}{Fig.}
\caption{\small\small\small\small\small{ Case II. $T^k\cU_{n,j}$ enters a u-gap  of $\cU^*$.}\label{fig4}}
\end{figure}

Based on the discussion in the above, we know that $\mu(T^{k}\cU^s_{n,j}\cap\cR^*)=0$ for all $k=1,\cdots,n-1$. i.e. $T^{n}\cU^s_{n,j}$ returns to $\cR^*$ for the first time.

We have thus proved that
$\cR^*=\cup_{n=1}^{\infty}\cR^s_n$, such that $\{\cR^s_n\}$ are disjoint $s$-subsets, with the property that $T^n\cR^s_n$ properly return to $\cR^*$ for the first time.
Thus we define $\cR_n^*=\cR^s_n$ almost surely for every $n\in\mathbb{N}$. This also implies that $\{\cR_n^*\}$ is a Markov partition of $\cR^*$.
\end{proof}

We have done with the construction of $\cR^*$ and the proof of Theorem \ref{main}.

{
\section{Advanced Statistical Properties due to Theorem~\ref{main}}\label{sec:stat}

Note that the partition \eqref{def Markov Partition} in Theorem~\ref{main}
could be regarded as a Young tower, whose base is the hyperbolic product set $\cR^*$.
As the tower heights are exactly the first return time to the base,
such Young tower is in fact isomorphic to the dynamical system $(M, T, \mu)$.
Furthermore, the first return times has exponential tail bounds (see \eqref{cRmbd}).

Under this structure, several advanced statistical properties have been established.

\subsection{Large/Moderate Deviation Principles}

We investigate the deviation of the Birkhoff average taking values away from its mean,
by investigating the decay rate of
\beqn
\mu\left\{x\in M: \ \left|\frac{1}{n} S_n f(x)- \mu(f) \right|>\varepsilon\right\}.
\eeqn
for a dynamically H\"older function $f\in H(\gamma)$.
Here $S_n f=\sum_{k=0}^{n-1} f\circ T^k$ and $\mu(f)=\int f d\mu$.
Also, we recall the Green-Kubo formula for variance,
that is, $\sigma^2=\sum\limits_{k=-\infty}^\infty \int f\cdot\, f\circ T^k\, d\mu$.

Under the exponential tail condition given in \eqref{cRmbd},
Rey-Bellet and Young showed in \cite{RBY08}  that the
logarithmic moment generating function
\beqn
\Psi(z)=\lim_{n\to \infty} \frac{1}{n} \log \int e^{zS_n f} d\mu
\eeqn
is analytic in the strip
$\{z\in \IC: \ |\Re(z)|\le \xi_{max}, \ |\Im(z)|\le \eta_{max}\}$
for some $\xi_{max}>0$ and $\eta_{max}>0$,
and then deduce the exponential large deviation principle.

\begin{theorem}[Exponential large deviations \cite{RBY08}]\label{thm: exp LD}
Let $f\in H(\gamma)$ and
let the rate function $I(t)$ be the Legendre transform of $\Psi(z)$. Then for any interval
$[t_1, t_2]\subset [\Psi'(-\xi_{max}), \Psi'(\xi_{max})]$,
\beqn
\lim_{n\to \infty} \frac{1}{n} \log \mu\left\{x\in \cM: \ \frac{1}{n} S_n f(x)\in [t_1, t_2]\right\}
=-\inf_{t\in [t_1, t_2]} I(t).
\eeqn
\end{theorem}
Note that this theorem is a ``local" large deviation result, as it only holds on
an interval containing the the mean $\mu(f)$.
Rey-Bellet and Young \cite{RBY08} further characterized the fluctuations of
$S_nf$ which are of an order intermediate
between $\sqrt n$ and $n$.

\begin{theorem}[Moderate deviations \cite{RBY08}]
Let $f\in H(\gamma)$ and
let $B_n$ be an increasing sequence of positive real numbers such that
$\lim_{n\to\infty} a_n/\sqrt n=\infty$ and $\lim_{n\to\infty} a_n/ n=0$. Then for any
interval $[t_1, t_2]\subset \IR$ we have
\beqn
\lim_{n\to \infty} \frac{1}{B_n^2/n} \log \mu\left\{x\in \cM: \ \frac{S_n f(x)-n\mu(f)}{B_n} \in [t_1, t_2]\right\}
=-\inf_{t\in [t_1, t_2]} \frac{t^2}{2\sigma^2}.
\eeqn
\end{theorem}

\subsection{Beyond the Classical Limit Theorems}

Reinforcements of the classical limit theorem, stated in Theorem~\ref{thm: CLT M},
are developed in several different directions.

\subsubsection{Almost Sure Limit Theorems}

Three methods - spectral methods, martingale methods, and induction arguments,
were applied by Chazottes and Gou\"ezel \cite{ChG07} to prove that
"whenever a limit theorem in the classical sense holds for a dynamical system,
a suitable almost-sure version also holds", under the assumption that
the tail bound of the first return times is of order $\cO(n^{-a})$ for some $a>2$.
It is apparent that our system $(M, T, \mu)$ (regarded as a Young tower) satisfies this assumption,
as it has exponential tail bound.
Combining with Theorem~\ref{thm: CLT M}, we obtain

\begin{proposition}[Almost sure central limit theorem]\label{ASCLT}
Let $f\in H(\gamma)$ such that $\mu(f)=0$.
Then for $\mu$-almost surely $x\in M$,
\beqn
\frac{1}{\log n} \sum_{k=1}^{n} \frac{1}{k} \delta_{S_k f(x)/\sqrt k}
\Longrightarrow \sigma N(0, 1), \ \ \text{as}\ n\to\infty.
\eeqn
Here ``$\Longrightarrow$" standards for the convergence in distribution.
\end{proposition}

\subsubsection{Almost Sure Invariance Principle}

The almost sure invariance principle is a functional version of
the central limit theorem, which ensures that the trajectories of a random process can be matched with the
trajectories of a Brownian motion with almost surely negligible error.
More precisely, let $\{X_n\}_{n\ge 0}$ be a real-valued process, which have
zero means and finite $(2+\epsilon)$-moments for some $\epsilon>0$.
Set $S_t=\sum_{n\le t} X_n$ for $t>0$.
We say that the process $\{X_n\}_{n\ge 0}$ satisfies an almost sure invariance principle (ASIP)
if there is $\lambda>0$ such that
\beq\label{ASIP}
\left|S_t - W_t \right| =\cO (t^{\frac12-\lambda}), \ \ \ \text{almost surely},
\eeq
where $W$ is a Brownian motion with variance $\lim_{n\to\infty} \text{Var}(S_n/\sqrt n)$.

Using a general theorem established by Philipp and Stout \cite{PhSt75}, we obtain

\begin{theorem}[Almost sure invariance principle]
Let $f\in H(\gamma)$ such that $\mu(f)=0$.
Then the process $\{f\circ T^n\}_{n\ge 0}$ satisfies the almost sure invariance principle.
\end{theorem}

The ASIP was first shown by Chernov \cite{C06} in the exponential case, and the vector-valued case
was later proved by Melbourne and Nicol \cite{MN09} using the martingale methods
and then by Gou\"ezel \cite{G10} with the purely spectral methods.
Under our assumptions (\textbf{H1})-(\textbf{H5}), a non-stationary version of ASIP was established by
the first and third authors in \cite{CYZ18}.

Note that the ASIP is the strongest form - it implies WIP, which in turn implies the CLT and the almost sure CLT.
The ASIP also implies many other limit laws (see \cite{PhSt75}).

\subsubsection{The Local Central Limit Theorem}

We say that a function $f: M\to \IR$ is periodic if there exist $c\in \IR$, $\lambda>0$,
$g: M\to \IZ$ and $h: M\to \IR$ measurable, such that $f=c+\lambda g+h-h\circ T$ almost surely.
Otherwise $f$ is said to be aperiodic, which implies in particular that $\sigma^2>0$.

The local central limit theorem (LCLT) concerns the convergence of the densities of
distributions $S_n f/\sqrt n$ to the Gaussian density.
The LCLT follows from the results by Gou\"ezel in \cite{G05}.

\begin{theorem}[Local Central Limit Theorem \cite{G05}]
Let $f\in H(\gamma)$ such that $\mu(f)=0$ and $f$ is aperiodic.
Then for any bounded interval $J\subset \IR$, for any real sequence $\kappa_n$
with $\kappa_n/\sqrt n\to \kappa\in \IR$, for any $u\in H(\gamma)$ and
for any $v: M\to \IR$ measurable, we have
\beqn
\sqrt n \mu\{x\in M:\ S_n f(x)\in J+\kappa_n+u(x)+v(T^nx)\}\to
|J|\varphi_\sigma(\kappa),
\eeqn
as $n\to \infty$, where $\varphi_\sigma(\kappa):=\dfrac{e^{-\kappa^2/(2\sigma^2)}}{\sigma\sqrt{2\pi}}$
is the density of the Gaussian distribution $\sigma N(0, 1)$.
\end{theorem}

Sz\'asz and Varj\'u \cite{SzV04} established
a version of the LCLT
in the exponential tail case for $\IR^d$-valued functions.
They further applied this result to prove that
the planar Lorentz process with a finite horizon is almost surely recurrent.

\subsection{Concentration Inequalities}

For our dynamical system $(M, T, \mu)$, a function
$K: M^n\to \IR$ is called separately Lipschitz if,
for any $0\le i\le n-1$, there exists a constant $\Lip_i(K)$ with
$$\left|K(x_0,\cdots, x_{i-1}, x_i,x_{i+1},\cdots, x_{n-1})-K(x_0,\cdots, x_{i-1}, x_i',x_{i+1},\cdots, x_{n-1})\right|
\le \Lip_i(K)|x_i-x_i'|$$
for all $x_0,\cdots, x_{n-1}, x_i'\in M$.

The scope of concentration inequalities is to understand to what extent
a function of a random process concentrates around its mean. To be precise,
we say that the system $(M, T, \mu)$
satisfies an exponential concentration inequality if there exists a constant $C>0$ such
that, for any $n\geq 1$ and any separately Lipschitz function $K: M^n\to \IR$,
one has
\beq\label{EPDC}
\int e^{K(x, T x, \cdots, T^{n-1}x)-\int K(y, T y, \cdots, T^{n-1}y) d\mu(y)} d\mu(x) \le e^{C\sum_{i=0}^{n-1} \Lip_i(K)^2}.
\eeq
Chazottes and Gou\"ezel \cite{CG12} obtained the exponential concentration inequalities
for Young towers with exponential tail.

\begin{theorem}[Concentration Inequalities \cite{CG12}]
The system $(M, T, \mu)$ satisfies an exponential concentration inequality.
\end{theorem}

}

\section{Proof of Theorem \ref{thm: equi measures}}\label{sec TF IS}

To establish thermodynamic formalism of towers of hyperbolic type,
Pesin, Senti and Zhang developed the framework of inducing schemes in \cite{PS08, PSZ08, PSZ16}. However, due to the fragmentation caused by countable singularities and unbounded differential of the map,  the Condition (22) used in \cite{PSZ16} fails for a large class of hyperbolic systems with singularities that we consider here. In   this section, we adapt the PSZ approach to study a class of equilibrium measures
for a dynamical system $(M, T)$ that satisfies Assumptions (\textbf{H1})-(\textbf{H5}),  under a new yet much weaker condition (\textbf{P*}).

More precisely, we consider the inducing scheme on the magnet $\cR^*$ constructed in Theorem \ref{main},
and the one-parameter family of geometric potentials $\varphi_t$ given by \eqref{geo potential}.
We shall prove Theorem \ref{thm: equi measures}, that is,
under a technical condition \eqref{cond left},
the equilibrium measure $\mu_t$ corresponding to the potential $\varphi_t$
exists and is unique for any $t$ in a neighborhood of 1.
Moreover, $\mu_t$ has exponential decay of correlations and satisfies the central limit theorem
with respect to a class of observables that include
bounded H\"older continuous functions.
In particular, $\mu_1=\mu$ is the mixing SRB measure in our Assumption (\textbf{H3}).

\subsection{Inducing Scheme Construction on $\cR^*$}\label{Sec: ISC}

Similar to what we did in Subsection \ref{sec symbolic coding},
we construct a symbolic coding for our final magnet $\cR^*$ in Theorem \ref{main}.
Recall that there is a countable family of mutually disjoint closed $s$-subsets $\cR_n^*\subset \cR^*$ such that
$\cR^*=\bigcup\limits_{n=1}^\infty \cR^*_n \pmod \mu$.\footnote{some $\cR^*_n$ may be empty.}
Moreover,
\begin{equation*}
n=\min\{k>0 \ | \ T^k\cR^*_n\cap \cR^*\ne \emptyset \pmod \mu\},
\end{equation*}
and $T^n\cR^*_n$ is a $u$-subset of $\cR^*$.
We can further decompose each non-trivial s-subset as $\cR^*_n=\bigcup_{j=1}^\infty \cR^*_{n,j}$.
We then denote the alphabet set by
\begin{equation*}
\cA=\{a=(n, j)\ : \ \mu(\cR^*_{n, j})\ne 0\},
\end{equation*}
and set
\begin{equation}\label{def tau}
\tau(x)=\tau(a):=n, \ \ \text{for any}\ a=(n,j)\in \cA, \  x\in \cR^*_a.
\end{equation}
The induced map $F: \cR^*\to \cR^*$ is then given by
$$
F(x)=T^{\tau(x)}(x), \ \ \text{for any}\ x\in \bigcup_{a\in \cA} \cR^*_a,
$$
and it is undefined on a null subset $\cN^*\subset \cR^*$.
Hence $F^{-1}\cR^*=\cR^*\backslash \cN^*$ and $F\cR^*=\cR^*\backslash F\cN^*$,
and thus for $k\ge 1$,
\beqn
F^{-k}\cR^*=\cR^*\backslash \bigcup_{j=0}^k F^{-j}\cN^*,
\ \text{and} \ \
F^k\cR^*=\cR^*\backslash \bigcup_{j=1}^k F^j\cN^*.
\eeqn
It is clear that $F$ preserves the SRB measure $\mu$.
Moreover, $F|_{\cR^*_a}$ can be extended to a homeomorphism of a neighborhood
of $\cR^*_a$,
as it is the restriction of the smooth map $T^{\tau(a)}$ on the solid rectangular hull $\cU(\cR^*_a)$.

Similar to \eqref{def pi},  we define
\begin{equation}\label{def pi 1}
\pi(\ab)=\bigcap_{k=-\infty}^\infty F^{-k} \cR^*_{a_k}.
\end{equation}
for any bi-infinite sequence $\ab=(a_k)_{k\in \IZ}\in \cA^\IZ$.
We notice that $\pi(\ab)$ is a single point
whose trajectory under $F$ lies in $ \bigcup_{a\in \cA} \cR^*_a = \cR^*\backslash \cN^*$,
and thus $\pi(\ab)$ belongs to the subset
\beqn\label{def cR star F}
\cR^*_F:=\bigcap_{k=-\infty}^\infty F^{-k}\left( \cR^*\backslash \cN^* \right)
=\cR^*\backslash \bigcup_{k=-\infty}^\infty F^k\cN^*
=\bigcup_{a\in \cA} \left(\cR_a^*\backslash \bigcup_{k=-\infty}^\infty F^k\cN^*\right)
=:\bigcup_{a\in \cA} \cR^*_{a, F}
\eeqn
of full $\mu$-measure.
Therefore, the map $F: \cR^*_F\to \cR^*_F$ is conjugate to
the two-sided full shift $\sigma: \cA^\IZ\to \cA^\IZ$ by the map $\pi$ in \eqref{def pi 1}.
Similar to Lemma \ref{lem period}, the periodic points of the system $(\cR^*_F, F)$
is dense in $\cR^*_F$.

\begin{remark}
The above construction fits the definition of inducing schemes that was introduced in \cite{PSZ16}, Section 2. More precisely, the map $T: M\to M$ admits an inducing scheme of hyperbolic type $\{\cA, \tau\}$, with inducing domain
$\cR^*_F=\cup_{a\in \cA} \cR^*_{a, F}$ and inducing time $\tau: \cR^*_F \to \IN$,
and satisfying Condition (I1), (I2) and  (I4) in \cite{PSZ16}.
We also point out that Condition (I3) holds automatically in our case,
since $\mu$ is a mixing SRB measure and hence the $\mu$-null set
$\bigcup_{k=-\infty}^\infty F^k\cN^*$ cannot support an invariant measure which
gives positive weight to any open subset.

For simplicity of notations, we still write $\cR^*$ and $\cR^*_a$ instead of $\cR^*_F$ and $\cR^*_{a, F}$ respectively in the rest of this section.
\end{remark}

\subsection{Preliminaries on Thermodynamics for Inducing Schemes}

In this subsection, we briefly review the main results by \cite{PSZ16}, under the setting for Markov systems with Markov partition .

\subsubsection{Terminology for Two-sided Countable Full Shift}

Let $S_{-k, \ell}$ be the collection of strings of the form $\cb=(c_{-k} \dots c_{\ell})$. Given $\cb\in S_{-k, \ell}$,
the corresponding cylinder is defined to be
\begin{equation}
[\cb]=\{\ab\in \cA^\IZ\,: \, a_i=c_i, \ i=-k, \dots, \ell\}.
\end{equation}
The $k$-variation of a function $\Phi: \cA^\IZ\to \IR$ is given by
\begin{equation}\label{def V_n}
V_n(\Phi)=\sup_{\cb\in S_{-n, n}}\ \sup_{\ab, \ab'\in [\cb]} \left\{\left|\Phi(\ab)-\Phi(\ab')\right| \right\}.
\end{equation}
$\Phi$ is called {\it{locally H\"older continuous}}\footnote{
Locally H\"older continuity implies the (strongly) summable variation property.
}
if there exists $C>0$ and $0<r<1$ such that,
\begin{equation}\label{loc holder}
V_n(\Phi)\le C r^n, \ \ n\ge 1.
\end{equation}
The Gurevich pressure of $\Phi$ is defined by the following form:
\begin{equation}
P_G(\Phi):=\lim_{n\to \infty} \frac{1}{n} \log \sum_{\sigma^n(\ab)=\ab}
\exp\left(\sum_{j=0}^{n-1}\Phi(\sigma^j(\ab))\right) \mathbb{I}_{[b]}(\ab)
\end{equation}
for any state $b\in \cA$, where $\mathbb{I}_{[b]}$ denotes the characteristic function of the cylinder $[b]$.
Note that $P_G(\Phi)$
exists whenever $\sum_{n\ge 1} V_n(\Phi)<\infty$ and it is independent of the choice of $b\in \cA$.
See \cite{Sarig99, Sarig03} for more details.

A measure $\nu_\Phi$ on $\cA^\IZ$ is a Gibbs measure for $\Phi$ if
there exists a constant $C>0$ such that for any $\cb\in S_{0, n-1}$ and any $\ab\in [\cb]$, we have:
$$
C^{-1}\le \dfrac{\nu_\Phi([\cb])}{\exp \left(-n P_G(\Phi) + \sum_{j=0}^{n-1}\Phi(\sigma^j(\ab))\right)} \le C.
$$
Let $\cM(\sigma)$ denote the set of $\sigma$-invariant ergodic Borel probability measures on $\cA^\IZ$
and
$$
\cM_\Phi(\sigma):=\{\nu\in \cM(\sigma)\,:\,\ \int \Phi d\nu>-\infty \}.
$$

 If $\sup_{\ab\in \cA^\IZ} \Phi(\ab)<\infty$ and $\sum_{n\ge 1} V_n(\Phi)<\infty$,
by Theorem 3.1 in \cite{PSZ16}, the variational principle holds.  That is, the pressure has another definition:
$$P_G(\Phi)=\sup_{\nu\in \cM_{\Phi}(\sigma)} \left\{ h_{\nu}(\sigma)+\int \Phi d\nu \right\},$$
where $h_{\nu}(\sigma)$ is the Kolmogorov-Sinai entropy of $\sigma$ with respect to the measure $\nu$.

A $\sigma$-invariant measure $\nu_\Phi\in \cM_\Phi(\sigma)$ is said to be an equilibrium measure for $\Phi$ provided  it obtains the above supremum, i.e.
$$
h_{\nu_\Phi}(\sigma)+\int \Phi d\nu_\Phi = P_G(\Phi).$$

By  \cite{PSZ16},   if $P_G(\Phi)<\infty$, there is a unique $\sigma$-invariant Gibbs measure $\nu_\Phi$. Furthermore,
if $h_{\nu_\Phi}(\sigma)<\infty$, then $\nu_\Phi$ is the unique equilibrium measure.

\subsubsection{Thermodynamics for Inducing Schemes}\label{sec thermo IS}

Given a function $\varphi: M\to \IR$,
the induced potential $\ovarphi: \cR^*\to \IR$ is defined by
\begin{equation}\label{def ovarphi}
\ovarphi(x)=\sum^{\tau(x)-1}_{k=0} \varphi(T^kx).
\end{equation}
For the rest of the paper, we make the following assumptions:\\

\noindent (\textbf{P1}) The induced potential $\ovarphi$ can be extended by continuity
to a function on $\cR^*_a$ for any $a\in \cA$;\\

We can then define the potential $\Phi:=\ovarphi\circ \pi$ on $\cA^\IZ$,
where $\pi$ is given by \eqref{def pi 1}.\\

\noindent (\textbf{P2}) $\Phi$ is locally H\"older continuous; \\

\noindent (\textbf{P3}) $\sum_{a\in \cA} \sup_{x\in \cR^*_a} \exp \ovarphi(x) <\infty$.\\

Denote the set of $T$-invariant ergodic Borel probability measure on $M$ by $\cM(T, M)$
and the set of $F$-invariant ergodic Borel probability measure on $\cR^*$ by $\cM(F, \cR^*)$.
For $\nu\in \cM(F, \cR^*)$, let
\begin{equation*}
Q_\nu:=\int_{\cR^*} \tau\ d\nu,
\end{equation*}
where $\tau$ is given by \eqref{def tau}.
If $Q_\nu<\infty$, then $\nu$ is lifted to a measure $\omega=\cL(\nu)\in \cM(T, M)$ given by
\begin{equation*}
\omega(E):=\frac{1}{Q_\nu} \sum_{a\in \cA} \sum^{\tau(a)-1}_{k=0} \nu(T^{-k}E\cap \cR^*_a), \ \
\text{for any Borel subset} \ E\subset M.
\end{equation*}
In such case, the measure $\omega\in \cM(T, M)$ is called liftable and $\nu=\cL^{-1}(\omega)$ is an induced measure for $\omega$.
We denote the class of liftable measures by $\cM_L(T, M)$.

Given a function $\varphi: M\to \IR$, we define:
\begin{equation}\label{lifted pressure}
P_L(\varphi):=\sup_{\omega\in \cM_L(T, M)} \{ h_\omega(T) + \int_M \varphi \ d\omega\},
\end{equation}
and $\omega_\varphi\in \cM_L(T, M)$ is called an equilibrium measure for $\varphi$ if it obtains the above supremum.

It was shown in \cite{PS08} that $-\infty<P_L(\varphi)<\infty$ if $P_G(\Phi)<\infty$.
To lift back to the system $(T, M)$, we also need the normalized induced potential defined as:
\begin{equation*}
\varphi^+:=\overline{\varphi-P_L(\varphi)}=\ovarphi-P_L(\varphi)\tau,
\end{equation*}
and let $\Phi^+:=\varphi^+\circ \pi$.
We suppose that:\\

\noindent (\textbf{P4}) There exists $\eps>0$ such that $\sum_{a\in \cA} \tau(a) \sup_{x\in \cR^*_a} \exp (\varphi^+(x)+\eps \tau(x)) <\infty$.\\

The following result is given by Theorem 4.4, 4.6 and 4.7 in \cite{PSZ16}.

\begin{proposition}\label{thm: PSZ16}
Assume that the potential $\varphi$ satisfies Conditions \textbf{(P1)-(P4)}.
Then
\begin{itemize}
\item[(a)] $P_G(\Phi^+)=0$ and $\sup_{\ab\in \cA^\IZ} \Phi^+(\ab)<\infty$, hence there
exists a unique $\sigma$-invariant Gibbs measure $\nu_{\Phi^+}$ for $\Phi^+$;
\item[(b)] $h_{\nu_{\Phi^+}}(\sigma)<\infty$, then the measure $\nu_{\Phi^+}$ is the unique equilibrium measure for $\Phi^+$.
Thus, $\nu_{\varphi^+}:=\pi_* \nu_{\Phi^+}$ is the unique $F$-invariant ergodic equilibrium measure for $\varphi^+$.
\item[(c)] $Q_{\nu_{\varphi^+}}<\infty$, then $\omega_{\varphi}=\cL(\nu_{\varphi^+})$ is the unique ergodic equilibrium measure in
$\cM_L(T, M)$.
\item[(d)] If $\nu_{\varphi^+}$ has exponential tail, that is, \\

\text{\noindent (\textbf{P5}) there are $C>0$ and $\theta\in (0, 1)$ such that $\nu_{\varphi^+}(\{x\in \cR^*: \ \tau(x)\ge n\})\le C\theta^n$,}\\

then the lifted measure $\omega_{\varphi}$ has exponential decay of correlations and satisfies the central limit theorem with respect to the class $\cH_\pi(r)$ of functions on $M$ given by
\begin{equation}\label{def cH pi}
\cH_\pi(r):=\{\varphi:\ \text{there is a constant}\ C_\varphi>0\ \text{such that}\
V_n(\varphi\circ \pi)\le C_\varphi r^n,\ \text{for any}\ n\ge 1\},
\end{equation}
for any $r\in (0, 1)$, where $V_n(\cdot)$ is given by \eqref{def V_n}.\footnote{
Note that $\varphi\circ \pi$ is locally H\"older continuous on $\cA^\IZ$ for any $\varphi\in \cH_\pi(r)$.
Also, $\bigcup_{r>0} \cH_\pi(r)$ contains all bounded H\"older continuous functions on $M$.
}
\end{itemize}
\end{proposition}

\subsection{The new condition (\textbf{P*}) for the Geometric Potentials}\label{sec TF M}
In this subsection, we concentrate on the geometric potential,
defined as \eqref{geo potential}. More precisely,  the geometric potentials $\varphi_t$ for $T:M\to M$ is given by
$$
\varphi_t(x)=-t\log |J^u T(x)|,  \ -\infty < t< \infty.
$$
The induced potential and normalized induced potential are then given by
\begin{equation*}
\ovarphi_t(x)=-t\log |J^u F(x)|,  \ \ \ \varphi^+_t(x)=\ovarphi_t(x) - P_L(\varphi_t) \tau(x).
\end{equation*}
Set $\Phi_t=\ovarphi_t\circ\pi$ and $\Phi^+_t=\varphi^+_t\circ \pi$.

We are now ready to state the technical condition for Theorem \ref{thm: equi measures}.\\
\\

\noindent{\textbf{(P*).}} There exist $t_0<1$, $C_0>0$ and $K_0>0$ such that
\begin{equation}\label{cond left}
\sum_{a\in \cA:\ \tau(a)=n}\ \  \sup_{x\in\cR^*_a} |J^u F(x)|^{-t_0}\le C_0 K_0^n.
\end{equation}

\vspace{0.1cm}

\begin{remark}
A sufficient condition for Condition \textbf{(P*)} is the following special case when $t_0=0$:
\beq\label{cond left 1}
\#\{a\in \cA:\ \tau(a)=n\}\le C_0 K_0^n,
\eeq
that is, the number of components $\cR_a^*$ with the first return time $\tau(a)=n$
grows at most exponentially in $n$.
This was the Condition (22) that used in \cite{PSZ16} to prove results for thermodynamical formalism.
In particular, \eqref{cond left 1} holds
if the norm of the derivative $DT$ is uniformly bounded.
In general, the derivative of $T: M\to M$ will blow up near singularity, which makes \eqref{cond left 1} fail.
Nevertheless,
Condition (\textbf{P*}) can be valid for some value $t_0\in (0, 1)$.
For instance, we shall show in Section \ref{sec: app} that
\eqref{cond left} holds when $t_0>1/2$
for the billiard maps associated with perturbations of
the periodic Lorentz gas.  Condition (\textbf{P*}) is essential for checking conditions \textbf{(P3)-(P5)}.
\end{remark}

\vspace{0.2cm}

Next we state and prove the two lemmas that will be used in the proof of Theorem \ref{thm: equi measures}.

\begin{lemma}\label{lem: check P1-2}
The potential $\varphi_t$ satisfies Conditions (\textbf{P1}) and (\textbf{P2}) for all $-\infty< t<\infty$.
\end{lemma}

\begin{proof}
By the construction in Subsection \ref{Sec: ISC}, it is clear that
$\ovarphi_t$ is a continuous function on $\cR^*_a$ for any $a\in \cA$, that is,
the potential $\varphi_t$ satisfies Condition (\textbf{P1}).

Next we check that $\varphi_t$ satisfies Condition (\textbf{P2}), that is, $\Phi_t$ is locally H\"older continuous.
Given $\cb=(c_{-n}, \dots, c_{-1}, c_0, c_1, \dots, c_{n})\in S_{-n, n}$ and $a, a'\in [\cb]$, set $x=\opi(a)$, $y=\opi(a')$ and $z=W^u(x)\cap W^s(y)$. Then $F^k(x)$ and $F^k(z)$ belong to the same unstable manifold in $\cR^*_{c_k}$,
for $0\le k\le n$.
By the bounded distortion condition in Assumption \textbf{(H4)}(2) and the properties of unstable manifolds,
\begin{eqnarray*}
\left| \log \dfrac{J^u F^{-1}( Fx)}{J^u F^{-1} (Fz)} \right|
\le \sum_{j=1}^{\tau(x)}  \left| \log \dfrac{J_W T^{-1}( T^j x)}{J_W T^{-1} (T^j z)} \right|
&\le & C_{\bJ} \sum_{j=1}^{\tau(x)} d_W(T^j x, T^j z)^{\bgamma_0} \\
&\le & C_{\bJ} C_1 \sum_{j=1}^{\tau(x)} \Lambda^{\bgamma_0(j-\tau(x))} d_W(F x, F z)^{\bgamma_0}\\
&\le & C_2 d_W(Fx, Fz)^{\bgamma_0},
\end{eqnarray*}
for some $C_2>0$. Furthermore, by bounded curvature property in \textbf{(H4)}(1),
\begin{align*}
d_W(F x, F z)&\le C_1 \Lambda^{-\tau(c_1)-\dots-\tau(c_n)} d_W(F^n x, F^n z)\\
&\le C_1 \Lambda^{-n}d_W(F^n x, F^n z) \le C_1 C_3\text{diam}(\cR^*) \Lambda^{-n}.
\end{align*}
for some $C_3>0$.
Thus, there is $C>0$ such that
$$
\left| \log \dfrac{J^u F^{-1}( Fx)}{J^u F^{-1} (Fz)} \right|\le C \Lambda^{-\bgamma_0 n}.
$$
Similarly, by properties of stable manifolds, we have
$$
\left| \log \dfrac{J^u F(y)}{J^u F(z)} \right|\le C \Lambda^{-\bgamma_0 n}.
$$
Therefore, $\Phi_t$ is locally H\"older continuous since
\begin{align*}
|\Phi_t(a)-\Phi_t(a')|&=|t| \left| \log \dfrac{J^u F(y)}{J^u F(x)} \right|\\
&\le
|t|\left| \log \dfrac{J^u F(y)}{J^u F(z)} \right| + |t|\left| \log \dfrac{J^u F^{-1}( Fx)}{J^u F^{-1} (Fz)}\right|
\le   2|t|C \Lambda^{-\bgamma_0 n}.
\end{align*}
\end{proof}

In the rest of this subsection, we show that

\begin{lemma}\label{lem: check P3-5}
Under condition (\textbf{P*}), there are two  numbers $\tb<1<\bt$ such that for all $t\in (\tb, \bt)$, there exists $c_t\in \IR$ so that
the potential $\varphi_t-c_t$ satisfies Conditions (\textbf{P3}), (\textbf{P4}) and (\textbf{P5}).
\end{lemma}

To prove Lemma \ref{lem: check P3-5}, we need the following sublemmas.

\begin{sublemma}\label{lem: PSZ pt}
Let $p_t:=P_L(\varphi_t)$. Then $p_1=0$, and there is a unique equilibrium measure $\mu_1$ on $M$.
Moreover, for all $t\in \IR$ one has
\begin{equation*}
p_t\ge (t-1) \int \varphi_1 d\mu_1.
\end{equation*}
\end{sublemma}

This sublemma was stated and proven in \cite{PSZ16}, Lemma 7.5, which holds under rather general conditions. So we will skip the proof here. In fact, the measure $\mu_1$ is exactly the mixing SRB measure $\mu$ for $T: M\to M$. Moreover, $\mu=\mu_1$ is lifted from the unique Gibbs measure $\nu_1$ for $\ovarphi_1$,
and hence for any $a\in \cA$ and any $x\in \cR^*_a$,
\begin{equation}\label{SRB J}
\mu(\cR^*_a)=\frac{1}{Q_{\nu_1}} \nu_1(\cR^*_a)\asymp \exp(-p_1 + \ovarphi_1(x))=|J^u F(x)|^{-1}.
\end{equation}
Let $h_\mu(T)$ be the entropy of $T: M\to M$ with respect to the SRB measure $\mu$,
then by Pesin's entropy formula (see \cite{P77}),
we have that $h_1:=h_\mu(T)=-\int \varphi_1 d\mu_1>0$.

We set
\begin{equation}\label{def Un}
J_a:=\sup_{x\in \cR^*_a} |J^u F(x)|, \ \ \text{and}\ \
U_n(t):=\sum_{a\in \cA:\ \tau(a)=n} J_a^{-t}.
\end{equation}
We allow $U_n(t)=\infty$ if the above sum diverges.
Note that for fixed $n$, the function $t\mapsto U_n(t)$ is non-increasing. Also,
$U_n(t)=0$ if the set $\{a\in \cA:\ \tau(a)=n\}$ is empty.
We then define a function $\kappa: (-\infty, \infty) \to [-\infty, \infty]$ by
\begin{equation}
\kappa(t):=\limsup_{n\to \infty} \frac{1}{n}\log U_n(t).
\end{equation}

\begin{sublemma}\label{lem: kappa}
Under assumption (\textbf{P*}), the function $\kappa$ has the following property:
\begin{itemize}
\item[(1)] The function $\kappa(t)$ is convex and non-increasing on $(-\infty, \infty)$;
\item[(2)] $\kappa(1)<0$ and $\kappa(t_0)\le \log K_0<\infty$, where $t_0$ is given by Condition \eqref{cond left};
\item[(3)] There are $\tb<1<\bt$ such that
\begin{equation}\label{cond kappa}
\kappa(t)<h_1(1-t), \ \ \text{for any}\ t\in (\tb, \bt).
\end{equation}
\end{itemize}
\end{sublemma}

\begin{proof} (1) Fix $n\in \IN$.
Given $t_1, t_2\in (-\infty, \infty)$ and $\alpha\in (0, 1)$, by H\"older inequality,
$$
U_n(\alpha t_1 + (1-\alpha) t_2)
=\sum_{a\in \cA:\ \tau(a)=n} \left[J_a^{- t_1}\right]^{\alpha} \left[J_a^{- t_2}\right]^{1-\alpha}
\le U_n(t_1)^\alpha U_n(t_2)^{1-\alpha}.
$$
Taking $\limsup\limits_{n\to\infty} \frac{1}{n}\log$ on both sides, we have that
$$
\kappa(\alpha t_1 + (1-\alpha) t_2)\le \alpha \kappa(t_1) + (1-\alpha) \kappa(t_2).
$$
Hence the function $\kappa$ is convex. As $t\mapsto U_n(t)$ is non-increasing for fixed $n$,
it is easy to see that $\kappa(t)$ is non-increasing as well.

(2) By \eqref{SRB J}, \eqref{def Un}, as well as \eqref{cRmbd} in Theorem \ref{main},
there is $C>0$ such that
$$
U_n(1)=\sum_{a\in \cA: \ \tau(a)=n} J_a^{-1}\le C\sum_{a\in \cA: \ \tau(a)=n} \mu(\cR^*_a)\le
C \mu(\cR^*_n)\leq CC_* \theta_*^{n}.
$$
Hence
$$
\kappa(1)=\limsup\limits_{n\to\infty} \frac{1}{n}\log U_n(1)\le \log\theta_*<0.
$$
On the other hand, (\textbf{P*}) implies that $U_n(t_0)\le C_0K_0^n$ and thus $\kappa(t_0)\le \log K_0<\infty$.

(3) If $\kappa\equiv -\infty$, then we simply take $\tb=-\infty$ and $\bt=\infty$.

Otherwise, one must have $\kappa(t)>-\infty$ for all $t\in (-\infty, \infty)$ due to convexity.
Let $\tb_0=\inf\{t\in \IR:\ \kappa(t)<\infty\}$, then $-\infty\le \tb_0\le t_0$.
By monotonicity, $\kappa(t)<\infty$ for all $t>\tb_0$.
Hence $\kappa$
is a real-valued convex function on the open interval $(\tb_0, \infty)$, and is thus continuous.

Take $\kappa_1(t)=\kappa(t)-h_1(1-t)$ for $t\in (\tb_0, \infty)$, then it is clear that $\kappa_1$ is continuous and
$\kappa_1(1)=\kappa(1)<0$.
By continuity, there are $\tb_0<\tb<1<\bt\le \infty$ such that
$\kappa_1(t)<0$ for any $t\in (\tb, \bt)$, from which \eqref{cond kappa} follows.
\end{proof}

Now we are ready to prove Lemma \ref{lem: check P3-5}.

\begin{proof}[Proof of Lemma \ref{lem: check P3-5}]
Let $\tb<1<\bt$ be given by Sublemma \ref{lem: kappa}. For any $t\in (\tb, \bt)$,
we set $c_t=h_1(2-t)$ and choose $\eps_t$ such that $0<\eps_t<h_1(1-t)-\kappa(t)$.
By Sublemma \ref{lem: kappa}, there is $D_t>0$ such that
$$
U_n(t)\le D_t e^{\kappa(t)n}\le  D_t e^{n[h_1(1-t)-\eps_t]}.
$$
Condition (\textbf{P3}) for the potential $\varphi_t-c_t$ reads that
$
\sum_{a\in \cA} J_a^{-t} e^{-c_t \tau(a)}<\infty,
$
which is satisfied since
$$
\sum_{a\in \cA} J_a^{-t} e^{-c_t \tau(a)}=\sum_{n=1}^\infty e^{-c_t n} U_n(t)
\le D_t \sum_{n=1}^\infty e^{-n(h_1+\eps_t)} <\infty.
$$
To verify Condition (\textbf{P4}) for the potential $\varphi_t-c_t$, we first notice that
$$
(\varphi_t-c_t)^+=\overline{\varphi_t-c_t}-P_L(\varphi_t-c_t)\tau
=\overline{\varphi_t}-c_t \tau-P_L(\varphi_t)\tau +c_t \tau=\overline{\varphi_t}-p_t\tau.
$$
Then we take $\eps=\eps_t/2$, and it suffices to show that
$$
\sum_{a\in \cA} \tau(a) \sup_{x\in \cR_a^*}\exp(\overline{\varphi_t}(x)-p_t\tau(x)+\eps\tau(x))
=\sum_{n=1}^\infty n e^{-n(p_t-\eps_t/2)} U_n(t)<\infty.
$$
Indeed, by Sublemma \ref{lem: PSZ pt}, one has
$$
\sum_{n=1}^\infty n e^{-n(p_t-\eps_t/2)} U_n(t)\le D_t\sum_{n=1}^\infty n e^{n[h_1(1-t)-p_t-\eps_t/2]}
\le D_t\sum_{n=1}^\infty n e^{-n\eps_t/2}<\infty.
$$
Finally we verify Condition (\textbf{P5}), that is, the Gibbs measure $\nu_{t}$ for the potential $(\varphi_t-c_t)^+=\overline{\varphi_t}-p_t\tau$ has exponential tail.
Note that the existence and uniqueness of $\nu_t$ is guaranteed
by Proposition \ref{thm: PSZ16}, as Condition (P1)-(P4) hold.
Moreover, $P_G((\varphi_t-c_t)^+\circ\opi)=0$ by Proposition \ref{thm: PSZ16} (a).
There is $C>0$ such that for any $a\in \cA$,
$$
\nu_t(\cR^*_a)\le C \sup_{x\in \cR^*_a}\exp((\varphi_t-c_t)^+(x))=C\ J_a^{-t} e^{-p_t\tau(a)}.
$$
Thus $\nu_{t}$ has exponential tail since
\begin{align*}
\nu_t(\{x\in \cR^*: \ \tau(x)\ge N\})&=\sum_{n\ge N} \sum_{a\in \cA: \ \tau(a)=n} \nu_t(\cR^*_a)
\le C\sum_{n\ge N} e^{-np_t} U_n(t) \\
&\le  CD_t \sum_{n\ge N} e^{n[h_1(1-t)-p_t-\eps_t]} \le  CD_t \sum_{n\ge N} e^{-n \eps_t} \le C_1 (e^{- \eps_t})^{N},
\end{align*}
for some constant $C_1>0$. This completes the proof of Lemma \ref{lem: check P3-5}.
\end{proof}

Now we are ready to prove the main Theorem \ref{thm: equi measures}.

\begin{proof}[Proof of Theorem \ref{thm: equi measures}]
Lemma \ref{lem: check P1-2} shows that Conditions (\textbf{P1}) and (\textbf{P2}) hold for all $\varphi_t$, and
Lemma \ref{lem: check P3-5}  shows that Conditions \textbf{(P3)-(P5)} hold for the potential $\varphi_t-c_t$
for all $t\in (\tb, \bt)$, where $c_t\in \IR$ is a constant that depends on $t$. Since $\varphi_t$ and $\varphi_t-c_t$ are cohomologous they admit the same equilibrium measures,
Theorem \ref{thm: equi measures} directly follows from Proposition \ref{thm: PSZ16}.
\end{proof}

\section{Applications to Perturbations of The Lorentz Gas}\label{sec: app}

In this section, we apply our results in previous sections to
the billiard maps associated with the periodic Lorentz gas,
also known as the Sinai dispersing billiards
(see~\cite{Sin70, GO74, SC87, CH96} and the references therein).
We also consider small perturbations of the Lorentz gas subject to external forces both during flights and at collisions,
whose ergodic properties were extensively studied in \cite{C01, C08, DZ11, Zh11, DZ13}.

\subsection{The Lorentz Gas and Its Perturbations}\label{sec: Lorentz setup}


We begin by fixing a billiard table
$\bQ:=\mathbb{T}^2 \setminus \bigcup_{i=1}^d \{\mbox{interior } \bGamma_i\}$
associated with a periodic Lorentz gas,
where $\bGamma_i$ are finitely many disjoint
scatterers placed inside the 2-torus $\mathbb{T}^2$.
Moreover, the scatters have $C^3$ boundaries with strictly positive curvature:
\begin{itemize}
\item[(i)] the $C^3$ norm of the boundary $\partial \bGamma_i$ in $\bQ$ is bounded above by
a constant $\bE_*>0$;
\item[(ii)] the curvature of the boundary $\partial \bQ$ is between $\bcK_*$
and $\bcK_*^{-1}$ for some $\bcK_*\in (0, 1)$.
\end{itemize}

We consider the dynamics of the billiard map on the table $\bQ$,
subject to external forces both during flight and at collisions.
Let $\bx=(\bq, \bp)\in \bOmega:=\cT \bQ$ be a phase point, where
$\bq$ is the position of a particle in the billiard table and
$\bp$ be the velocity vector.
For a $C^2$ stationary external force
$\mathbf{F}: \mathbb{T}^2 \times \mathbb{R}^2 \to \mathbb R^2$, the perturbed billiard flow	
${\bf \Phi}^t$ satisfies the following differential equation between collisions:
\begin{equation}\label{flowf}
    \frac{d \bq}{dt} =\bp(t) , \qquad
    \frac{d \bp}{dt} = \mathbf{F}(\bq, \bp) .
\end{equation}
At collision, the trajectory experiences possibly nonelastic reflections with slipping along
the boundary:
\begin{equation}\label{reflectiong}
(\bq^+(t_i), \bp^+(t_i)) = (\bq^-(t_i), \bR  \bp^-(t_i)) +\mathbf G(\bq^-(t_i), \bp^-(t_i)),
\end{equation}
where $\bR \bp^-(t_i)= \bp^-(t_i)+2(\bn(\bq^-)\cdot \bp^{-})\bn(\bq^-))$ is the usual reflection operator,
$\bn(\bq)$ is the unit normal vector to the billiard wall $\partial \bQ$ at
$\bq$ pointing inside the table,
and $\bq^-(t_i), \bp^-(t_i)$, $\bq^+(t_i)$ and $\bp^+(t_i)$ refer to the incoming and
outgoing position and velocity vectors, respectively.
$\mathbf G$ is an external force acting on the incoming trajectories.
Note that we allow $\bf G$ to change both the position and velocity of the particle at the moment
of collision. The change in velocity can be thought of as a kick or twist while a change in
position can model a slip along the boundary at collision.

In \cite{C01, C08}, Chernov considered billiards under small external forces $\mathbf{F}$
with $\mathbf G=0$, and $\mathbf{F}$ to be stationary.
In \cite{Zh11} a twist force was considered assuming $\mathbf{F}=0$ and
$\bG$ depending on and affecting only the velocity, not the position.
Here we consider a combination of these two cases for systems
under more general forces $\mathbf{F}$ and $\mathbf{G}$.
We make the following four assumptions, combining those in \cite{C01, Zh11}. \\

\noindent\textbf{(A1)} (\textbf{Invariant space}) \emph{Assume the dynamics preserve
a smooth function $\bcE(\bq,\bp)$ on $\bOmega$.
Its level surface $\bOmega_c:=\bcE^{-1}(c)$, for any $c>0$,
is a compact 3-d manifold such that $\|\bp\|>0$ on $\bOmega_c$ and
for each $\bq\in \bQ$ and $\bp\in S^1$ the ray $\{(\bq, t\bp), t>0\}$
intersects the manifold $\bOmega_c$ in exactly one point.} \\

Assumption ({\bf{A1}}) specifies an additional integral of motion,
so that we only consider restricted systems on a compact level surface $\bOmega_c$.
Under this assumption the particle will not become overheated, and its speed will remain bounded.
For any phase point $\bx \in \bOmega$ for the flow,
let $\btau(\bx)$ be the length of the  trajectory between $\bx$ and its next non-tangential collision. \\

\noindent(\textbf{A2}) (\textbf{Finite horizon}) \emph{
There exists $\btau_*>0$ such that free paths between successive reflections are uniformly bounded, i.e.,
$\btau_* \le \btau( \bx ) \le  \btau_*^{-1}$ for any  $\bx \in \Omega. $
}\\

\noindent(\textbf{A3}) (\textbf{Smallness of the perturbation}).
\emph{We assume there exists $\beps_*>0$ small enough, such that
$\|\mathbf{F}\|_{C^1}<\beps_*$ and $\|\mathbf G\|_{C^1}<\beps_*.$
}\\

For some fixed $c>0$. We denote $T_{\bF,\bG}: M\to M$ as the billiard map
associated to the billiard flow ${\bf \Phi}^t$ on the compact level surface $\bOmega_c$,
where $M$ is the collision space containing all post-collision vectors
based at the boundary of the billiard table $\bQ$.
More explicitly, we set $M = \bigcup_{i=1}^d I_i \times [-\pi/2, \pi/2]$,
where each $I_i$ is an interval with endpoints identified and
$|I_i| = |\partial \bGamma_i|$, $i=1, \dots,  d$.
The phase space $M$ for the billiard map $T_{\bF, \bG}$
is then parametrized by the canonical coordinates $(\br, \bvf)$,
where $\br$ represents the arclength parameter on the boundaries of the scatterers (oriented clockwise),
and $\bvf$ represents the angle an outgoing trajectory makes with the
unit normal to the boundary at the point of collision.\\

\noindent(\textbf{A4})  \emph{We assume both forces $\bF$ and $\bG$ are stationary
and that $\bG$ preserves tangential collisions. Moreover, we assume that $\bF$ and $\bG$ are chosen such that the perturbed system is time-reversible.}\\

The case $\mathbf{F} = \mathbf{G}=0$ corresponds to
the classical billiard dynamics for the periodic Lorentz gas.
It preserves the kinetic energy $\bcE =\frac{1}{2} \|\bp\|^2$.
We denote by $\cF(\bQ, \btau_*, \beps_*)$ the class of all perturbed billiard maps
defined by the dynamics \eqref{flowf} and \eqref{reflectiong} under
forces $\mathbf{F}$ and $\mathbf{G}$, satisfying assumptions {\bf (A1)}-{\bf (A4)}.

\subsection{Proof of Theorem \ref{thm: C1}}

The fact that Assumptions {\bf (H1)}-{\bf (H5)} hold for any perturbed billiard map
$T\in \cF(\bQ, \btau_*, \beps_*)$ was established in \cite{DZ11, DZ13}.
In fact, for Assumptions \textbf{(H3)}, only the existence and finitude of ergodic SRB measures
were proven by the spectral analytic approach in \cite{DZ11, DZ13}.
Nevertheless, as each of the SRB measures is mixing up to a cyclic permutation,
we can always assume that there is a mixing SRB measure $\mu$, by taking a higher power of $T$.

It remains to show that Condition \eqref{cond left} holds when $t_0\in (1/2, 1)$, that is,

\begin{lemma}\label{lem: Lorentz cond left}
For sufficiently small $\beps_*>0$,
the perturbed system $T\in \cF(\bQ, \btau_*, \beps_*)$
satisfies Condition (\textbf{P*}) given by \eqref{cond left} for any $t_0\in (1/2, 1)$.
\end{lemma}

Before we prove Lemma \ref{lem: Lorentz cond left},
we first recall some facts about the hyperbolicity and singularities
for the perturbed billiard systems.
As introduced in the previous subsection \ref{sec: Lorentz setup},
the phase space for $T\in \cF(\bQ, \btau_*, \beps_*)$ is given by
$M = \bigcup_{i=1}^d I_i \times [ - \frac{\pi}{2}, \frac{\pi}{2} ]$.
We define the set $\cS_0 = \{ \bvf = \pm \frac{\pi}{2} \}$ and
$\cS_{\pm n} = \cup_{i = 0}^n T^{\mp i} \cS_{0}$.
For a fixed $\bk_0 \in \mathbb{N}$,
we define for $\bk \geq \bk_0$, the homogeneity strips,
\beq\label{homogeneity}
\IH_{\bk} = \{(\br,\bvf) : \pi/2 - \bk^{-2} < \bvf < \pi/2 - (\bk+1)^2 \}.
\eeq
The strips $\IH_{-\bk}$ are defined similarly near $\bvf = -\pi/2$.  We also define
$$\IH_0 = \{ (\br, \bvf) : -\pi/2 + \bk_0^{-2} < \bvf < \pi/2 - \bk_0^{-2} \}.$$
The set $\cS_{0,H} = \cS_0 \cup (\cup_{|\bk| \ge \bk_0} \partial \IH_{\pm \bk} )$ is therefore fixed
and will give rise to the singularity sets for the map $T\in \cF(\bQ, \btau_*, \beps_*)$.
We define
$\cS_{\pm n}^H = \cup_{i = 0}^n T^{\mp i} \cS_{0,H}$ to be the singularity sets for
$T^{\pm n}$, $n \ge 0$.
Note that if $(\br, \bvf)\in \IH_{\pm \bk}$, then $\cos\bvf\le \bk^{-2}$.

For any $n\in \IN$,
we denote the image of $\bx=(\br, \bvf)$ under $T^n$ by
$\bx_n=T^n\bx=(\br_n, \bvf_n)$.
Let $\|\cdot\|$ be the norm induced by the Euclidean metric on $\Omega=\cT \bQ$.
It was shown in \cite{DZ13} that under an adapted norm $\|\cdot\|_*$,
the Jacobian of $T$ along the unstable manifold satisfies that
$$
\dfrac{\|d\bx_1\|_*}{\|d \bx\|_*}\ge 1+\frac13 \btau_*\bcK_*=:\bLambda_*>1.
$$
Near grazing collisions, there are $\bB_1>0$ and $\bB_2>0$, which depends on $\bQ$ and  $\btau_*$
such that
$$
\dfrac{\bB_1(1-\bB_2\beps_*)}{\cos \bvf_1}\le \dfrac{\|d\bx_1\|_*}{\|d \bx\|_*}\le \dfrac{\bB_1(1+\bB_2\beps_*)}{\cos \bvf_1}.
$$
Therefore, for sufficiently small $\beps_*$, there is $\bL_0>0$ such that
\begin{equation}\label{Jac cos bound 1}
\dfrac{\|d\bx\|_*}{\|d \bx_1\|_*}\le
\begin{cases}
 \bLambda_*^{-1}, \ & \ \bx_1\in \IH_0, \\
 \bL_0 \cos\bvf_1\le \bL_0\bk^{-2}, \ & \ \bx_1\in \IH_{\pm \bk}\ \text{for}\ \bk\ge \bk_0,
\end{cases}
\end{equation}
that is, $\dfrac{\|d\bx\|_*}{\|d \bx_1\|_*}\le \bD_\bk$ if $\bx_1\in \IH_{\pm \bk}$, where $\bD_\bk$ is given by
\begin{equation}\label{Jac cos bound 2}
\bD_\bk=
\begin{cases}
\bLambda_*^{-1}, \ & \ \bk=0, \\
\bL_0\bk^{-2}, \ & \ \bk\ge \bk_0.
\end{cases}
\end{equation}

We are now ready to prove Lemma \ref{lem: Lorentz cond left}.

\begin{proof}[Proof of Lemma \ref{lem: Lorentz cond left}]
According to the structure of singular curves and finite horizon property (\textbf{A2}),
there is $\bdelta_0>0$ such that for any unstable manifold $W$ with length $|W|\le \bdelta_0$,
the image $TW$ can be cut by $\cS_0$ into at most
$
\dfrac{\max \btau(x)}{\min\btau(x)}\le \btau_*^{-2}
$
components (see \cite[\S 5.10]{CM}).
Each component is a strictly increasing curve in some $I_i \times [ - \frac{\pi}{2}, \frac{\pi}{2} ]$,
and can be further cut by each homogeneous singular curve in $\cS_{0, H}$ at most once.

By induction and  Assumption (\textbf{H4})(1), since any unstable manifold $W$ has length no greater than $\bc_M$,
$T^nW$ can be cut by $\cS_{n-1}$ into at most
$
\left(\bc_M \btau_*^{-2}/\bdelta_0\right)^n
$
increasing curves, each of which can be further cut by each homogeneous curve in $\cS_{n-1}^H$ at most once.
Therefore, by the construction of $\cR^*_a$ in Theorem \ref{main},
for any $a\in \cA$ with $\tau(a)=n$, the cardinality of the following set
\beqn
\{a'\in \cA:\ \tau(a')=n, \ \text{and} \ T^j\cR^*_{a'} \ \text{and}\ T^j\cR^*_{a}\ \text{are in the same}\ \IH_{\bk} \ \text{for all}\ 1\le j\le n\}
\eeqn
is no greater than $\left(\bc_M \btau_*^{-2}/\bdelta_0\right)^n$.
Thus, for any $t_0\in (1/2, 1)$, by \eqref{Jac cos bound 1} and \eqref{Jac cos bound 2}, we have
\begin{eqnarray*}
\sum_{a\in \cA:\ \tau(a)=n}\ \  \sup_{\bx\in\cR^*_a} |J^u F(\bx)|_*^{-t_0}
&\le & \sum_{a\in \cA:\ \tau(a)=n}\ \  \sup_{\bx\in\cR^*_a}  \left( \dfrac{\|d\bx\|_*}{\|d \bx_1\|_*}
\dfrac{\|d\bx_1\|_*}{\|d \bx_2\|_*} \dots \dfrac{\|d\bx_{n-1}\|_*}{\|d \bx_n\|_*} \right)^{t_0} \\
&\le & \left(\bc_M \btau_*^{-2}/\bdelta_0\right)^n \sum_{(\bk_1, \dots, \bk_n):\ \bx_i\in \IH^{\bk_i}}
\left( \bD_{\bk_1} \dots \bD_{\bk_n} \right)^{t_0} \\
&\le & \left(\bc_M \btau_*^{-2}/\bdelta_0\right)^n \left( \sum_{\bk} \bD_\bk \right)^n\\
&\le & \left(\bc_M \btau_*^{-2}/\bdelta_0\right)^n \left(\bLambda_*^{-t_0}+\bL_0^{t_0} \sum_{|\bk|\ge\bk_0}\bk^{-2t_0}\right)^n=: K_0^n.
\end{eqnarray*}
Transferring back to the Euclidean metric, we have that
$$\sum\limits_{a\in \cA:\ \tau(a)=n}\ \  \sup\limits_{\bx\in\cR^*_a} |J^u F(\bx)|^{-t_0}\le C_0 K_0^n$$
for some $C_0>0$. This finishes the proof of Lemma \ref{lem: Lorentz cond left}.
\end{proof}

\section*{Acknowledgements}
F.~Wang is supported by the State Scholarship Fund from China Scholarship Council (CSC).
H.-K.~Zhang is partially supported by the NSF Career Award (DMS-1151762).


\begin{thebibliography}{BSC90}



\bibitem{ABV}  A.~Alves, C.~Bonatti, and M.~Viana.
\emph{SRB measures for partially hyperbolic systems whose central direction is
mostly expanding},
{Invent. Math.} \textbf{140} (2000), 351--398.


\bibitem{ADU93} J.~Aaronson, M.~Denker, and M.~Urba\'nski.
\emph{Ergodic theory for Markov fibred systems and parabolic rational maps},
{Trans. Am. Math. Soc.} \textbf{337} (1993), 495--548.

\bibitem{BG06} P.~B\'alint and S.~Gou\"ezel.
\emph{Limit theorems in the stadium billiard},
{Comm. Math. Phys.} \textbf{263} (2006), no. 2, 461--512.

\bibitem{BCD11} P.~B\'alint, N.~I.~Chernov, and D.~Dolgopyat.
\emph{Limit theorems for dispersing billiards with cusps},
{Comm. Math. Phys.} \textbf{308} (2011), no. 2, 479--510.



\bibitem{B75} R.~Bowen.
\emph{Equilibrium States and the Ergodic Theory of Axiom A Diffeomorphisms},
{Lect. Notes in Math.} \textbf{470}, Springer-Verlag, Berlin-New York, 1975.



\bibitem{BS80}
L.~A.~Bunimovich and Ya.~G.~Sinai
\emph{Markov partitions for dispersing billiards},
{Commun. Math. Phys.} \textbf{73} (1980), 247--280.


\bibitem{BSC90}
L.~A.~Bunimovich, Ya.~G.~Sinai, and N.~I.~Chernov.
\emph{Markov partitions for two-dimensional hyperbolic billiards},
{Russian Math. Surveys} \textbf{45} (1990), 105--152.

\bibitem{BSC91}
L.~A.~Bunimovich, Ya.~G.~Sinai, and N.~I.~Chernov.
\emph{Statistical properties of two-dimensional hyperbolic billiards},
{Russian Math. Surveys}. \textbf{46} (1991), 47--106.

\bibitem{Buzzi-Sarig03}
J.~Buzzi and O.~Sarig.
\emph{Uniqueness of equilibrium measures for countable Markov shifts
and multidimensional piecewise expanding maps},
{Ergod. Th. Dynam. Syst.} \textbf{23} (2003), 1383--1400.






\bibitem{ChG07}
J.-R. Chazottes and S. Gou\"ezel.
\emph{On almost-sure versions of classical limit theorems for dynamical systems},
{Probab. Th. \& Rel. Fields.} \textbf{138} (2007), no. 1-2, 195--234.


\bibitem{CYZ18}
J. Chen, Y. Yang and H-K Zhang.
\emph{Non-stationary Almost Sure Invariance Principle for Hyperbolic Systems with Singularities}
{J. Stat. Phys.} \textbf{172} (2018), no. 6, 1499--1524.



\bibitem{CG12} J.-R. Chazottes and S. Gou\"ezel,
\emph{Optimal Concentration Inequalities for Dynamical Systems},
{Commun. Math. Phys.} \textbf{316} (2012), 843--889.


\bibitem{C99} N.~I.~Chernov.
\emph{Decay of correlations in dispersing billiards},
{J. Statist. Phys.} \textbf{94} (1999), 513--556.

\bibitem{C01} N.~I.~Chernov.
\emph{Sinai billiards under small external forces.}
{Ann. Henri Poincare}, \textbf{2} (2001), 197--236.

\bibitem{C06} N.~I.~Chernov.
\emph{Advanced statistical properties of dispersing billiards},
{J. Statist. Phys.} \textbf{122} (2006), 1061--1094.

\bibitem{C08} N.~I.~Chernov.
\emph{Sinai billiards under small external forces II.}
{Ann. Henri Poincare}, \textbf{9} (2008), 91--107.

\bibitem{CD} N.~I.~Chernov and D.~Dolgopyat.
\emph{Brownian Brownian Motion-I},
{Memoirs of AMS.} \textbf{198} (2009).

\bibitem{CD09a}
N.~I.~Chernov and D.~Dolgopyat.
\emph{The Galton board: limit theorems and recurrence},
{J. Amer. Math. Soc.} \textbf{22} (2009), no. 3, 821--858.

\bibitem{CD09b} N. I. Chernov and D. Dolgopyat.
\emph{Anomalous current in periodic Lorentz gases with infinite horizon},
{Russian Math. Surveys} \textbf{64} (2009), 651--699.

\bibitem{CD10} N. I. Chernov and D. Dolgopyat.
\emph{Lorentz gas with thermostatted walls},
{Ann. Henri Poincar\'e} \textbf{11} (2010), no. 6, 1117--1169.




\bibitem{CH96}  N.~I.~Chernov and  C.~Haskell.
\emph{Nonuniformly hyperbolic K-systems are Bernoulli},
{Ergod. Th. Dynam. Sys.} \textbf{16} (1996), 19--44.

\bibitem{CM}
N.~I.~Chernov, and R.~Markarian.
\emph{Chaotic Billiards},
{Math. Surveys Monographs} \textbf{127}, AMS, Providence, 2006.


\bibitem{CZ05a}
N.~I.~Chernov and H.-K.~Zhang.
\emph{Billiards with polynomial mixing rates},
{Nonlineartity} \textbf{4} (2005), 1527--1553.






\bibitem{CZ09}
N.~I.~Chernov and H.-K.~Zhang.
\emph{On statistical properties of hyperbolic systems with singularities},
{J. Statist. Phys.} \textbf{136} (2009), 615--642.





\bibitem{DZ11} M.~Demers and H.-K.~Zhang.
\emph{Spectral analysis of the transfer operator for the Lorentz gas},
{J. Modern Dynamics} \textbf{5} (2011), 665--709.

\bibitem{DLV} M. Demers, C. Liverani and V. Baladi, \emph{Exponential decay of correlations for finite horizon Sinai billiard flows} Invent. Math. 211 39-177 (2018).

\bibitem{DZ13} M.~Demers and H.-K.~Zhang.
\emph{A functional analytic approach to perturbations of the Lorentz gas},
{Commun. Math. Phys.}  \textbf{324} (2013), 767--830.



\bibitem{Dob68} R.~Dobrushin.
\emph{Description of a random field by means of conditional probabilities
and conditions for its regularity},
{Teor. Veroyatnoistei i Primenenia} \textbf{13} (1968), 201--229;
{English translation: Theory of Prob. and Appl.} \textbf{13} (1968), 197--223.


\bibitem{D01} D.~Dolgopyat.
\emph{On dynamics of mostly contracting diffeomorphisms},
{Commun. Math. Phys.} \textbf{213} (2001), 181--201.

\bibitem{DSV08} D.~Dolgopyat, D.~Sz\'asz, and T.~Varj\'u.
\emph{Recurrence properties of planar Lorentz process},
{Duke Math. J.} \textbf{142} (2008), no. 2, 241--281.

\bibitem{DSV09} D.~Dolgopyat, D.~Sz\'asz, and T.~Varj\'u.
\emph{Limit theorems for locally perturbed planar Lorentz processes},
{Duke Math. J.} \textbf{148} (2009), no. 3, 459--499.












\bibitem{GO74}   G.~Gallavotti and D.~Ornstein.
\emph{Billiards and Bernoulli schemes},
{Commun. Math. Phys.} \textbf{38} (1974), 83--101.




\bibitem{G05} S. Gou\"ezel.
\emph{Berry-Esseen theorem and local limit theorem for non uniformly expanding maps},
{Annales de l'IHP Probabilit\'es et Statistiques} \textbf{41} (2005), 997--1024.


\bibitem{G10} S. Gou\"ezel.
\emph{Almost sure invariance principle for dynamical systems by spectral methods},
{Ann. Probab.} \textbf{38} (2010), no. 4, 1639--1671.


\bibitem{Gurevich69}
B.~M.~Gurevich.
\emph{Topological entropy for denumerable Markov chains},
{Dokl. Acad. Nauk SSSR} \textbf{187} (1969);
{English translation: Soviet Math. Dokl.} \textbf{10} (1969), 911--915.

\bibitem{Gurevich70}
B.~M.~Gurevich.
\emph{Shift entropy and Markov measures in the path space of a denumerable graph},
{Dokl. Acad. Nauk SSSR} \textbf{192} (1970);
{English translation: Soviet Math. Dokl.} \textbf{11} (1970), 744--747.

\bibitem{KS86} A.~Katok and J.~M.~Strelcyn.
\emph{Invariant Manifolds, Entropy and Billiards;
Smooth Maps with Singularities},
{Lect. Notes Math.} \textbf{1222}, Springer, New York 1986.

\bibitem{LanRu69} O.~Lanford and D.~Ruelle.
\emph{Observables at infinity and states with short range correlations in statistical mechanics},
{Commun. Math. Phys.} \textbf{13} (1969), 194--215.



\bibitem{MN09} I.~Melbourne and M.~Nicol.
\emph{A vector-valued almost sure invariance principle for hyperbolic dynamical systems},
{Ann. Probab.} \textbf{37} (2009), no. 2, 478--505.




\bibitem{P77} Ya.~Pesin.
\emph{Characteristic Lyapunov exponents, and smooth ergodic theory},
{Russ. Math. Surv.} \textbf{32} (1977), no. 4, 55-114.

\bibitem{P92} Ya.~Pesin.
\emph{Dynamical systems with generalized hyperbolic attractors: hyperbolic,
ergodic and topological properties},
{Ergod. Th. Dynam. Syst.} \textbf{12} (1992), 123--152.

\bibitem{PS08} Ya.~Pesin.
\emph{Equilibrium measures for maps with inducing schemes},
{J. Mod. Dyn.} \textbf{2} (2008), no. 3, 397--430.

\bibitem{PSZ08} Ya.~Pesin.
\emph{Lifting measures to inducing schemes},
{Ergod. Th. Dynam. Syst.} \textbf{28} (2008), no. 2, 553--574.

\bibitem{PSZ16} Ya.~Pesin, S.~Senti, and K.~Zhang.
 \emph{Thermodynamics of towers of hyperbolic type},
{Trans. Amer. Math. Soc.}  \textbf{368} (2016), no. 12,  8519--8552.

\bibitem{PhSt75}
W. Philipp and W. Stout.
\emph{Almost sure invariance principles for partial sums of weakly dependent
random variables},
{Memoir. Amer. Math. Soc.} \textbf{161}: (1975).

\bibitem{RBY08}
L.~Rey-Bellet and L.-S.~Young.
\emph{Large deviations in non-uniformly hyperbolic dynamical systems},
{Ergod. Th. Dynam. Syst.} \textbf{28} (2008), no. 2, 587--612.


\bibitem{Ru78} D.~Ruelle.
\emph{Thermodynamic formalism},
{Encyclopedia of Mathematics and its Applications},
vol 5,  Addison-Wesley, 1978.


\bibitem{Sataev92} E.~Sataev.
\emph{Invariant measures for hyperbolic maps with singularities},
{Russ. Math. Surv.} \textbf{47} (1992), 191--251.


\bibitem{Sin70} Ya.~G.~Sinai.
\emph{Dynamical systems with elastic reflections.
Ergodic properties of dispersing billiards},
{Russ. Math. Surv.} \textbf{25} (1970), 137--189.

\bibitem{Sin72} Ya.~G.~Sinai.
\emph{Gibbs measures in ergodic theory},
{Russ. Math. Surv.} \textbf{27} (1972), 21--69.

\bibitem{SC87}  Ya.~G.~Sinai and  N.~I.~Chernov.
\emph{Ergodic properties of some systems of two-dimensional discs and
three-dimensional spheres},
{Russ. Math. Surv.} \textbf{42} (1987), 181--207.




\bibitem{Sarig99} O.~Sarig.
\emph{Thermodynamic formalism for countable Markov shifts},
{Ergod. Th. Dynam. Syst.} \textbf{19} (1999), no. 6, 1565--1593.

\bibitem{Sarig01} O.~Sarig.
\emph{Thermodynamic formalism for null recurrent potentials},
{Israel J. Math.} \textbf{121} (2001), 285--311.


\bibitem{Sarig03} O.~Sarig.
\emph{Characterization of existence of Gibbs measures for countable Markov shifts},
{Proc. Amer. Math. Soc.} \textbf{131:6} (2003), 1751--1758.


\bibitem{SYZ} M.~Stenlund, L.~S.~Young, and H.-K.~Zhang.
\emph{Dispersing billiards with moving scatterers},
{Commun. Math. Phys.} \textbf{322} (2013) 909--955.

\bibitem{SzV04}  D. Sz\'asz and T. Varj\'u.
\emph{Local limit theorem for Lorentz process and its recurrence in the plane,}
{Ergodic Theory Dynam. Systems} \textbf{24} (2004), 257--278.


\bibitem{VZ16} S.~Vaienti and H.-K.~Zhang.
\emph{Optimal bounds on correlation decay rates for nonuniform hyperbolic systems},
2016, submitted.


\bibitem{VJ62} D.~Vere-Jones.
\emph{Geometric ergodicity in denumerable Markov chains},
{Quart. J. Math. Oxford} \textbf{13:2} (1962), 7--28.


\bibitem{Y98} L.~S.~Young.
\emph{Statistical properties of systems with some hyperbolicity including certain billiards},
{ Ann. Math.} \textbf{147} (1998), 585--650.

\bibitem{Y99}  L.~S.~Young.
\emph{Recurrence times and rates of mixing},
{Israel J. Math.} \textbf{110} (1999), 153--188.


\bibitem{Zh11} H.-K.~Zhang.
\emph{Current in periodic Lorentz gases with twists.}
{Commun. Math. Phys.} \textbf{306} (2011), 747--776.



\end{thebibliography}
\end{document}